\documentclass{amsart}

\usepackage{amsmath, amsthm, amssymb, enumerate, verbatim}
\newcommand{\pp}[1]{\ensuremath{\stackrel{#1}{\cdot}}} 
\newcommand{\bsg}{Balog-Szemer\'{e}di-Gowers} 
\newcommand{\ST}{Szemer\'{e}di-Trotter} 
\DeclareMathOperator{\sym}{Sym}
\DeclareMathOperator{\stab}{Stab}
\DeclareMathOperator{\trans}{Trans}
\newcommand{\ceil}[1]{\ensuremath{\lceil#1\rceil}} 
\newcommand{\floor}[1]{\ensuremath{\lfloor#1\rfloor}} 
\newcommand{\aff}[2]{\ensuremath{\mathrm{Aff}(#1,#2)}}
\newcommand{\C}{\ensuremath{\mathbb{C}}}
\newcommand{\F}{\ensuremath{\mathbb{F}}}
\newcommand{\Q}{\ensuremath{\mathbb{Q}}}
\newcommand{\R}{\ensuremath{\mathbb{R}}}
\newcommand{\Z}{\ensuremath{\mathbb{Z}}}
\newcommand{\andd}{\ensuremath{\qquad\mbox{and}\qquad}} 
\renewcommand{\epsilon}{\varepsilon}
\newcommand{\actson}{\ensuremath{\curvearrowright}}

\makeatletter
\newtheorem*{rep@theorem}{\rep@title}
\newcommand{\newreptheorem}[2]{%
\newenvironment{rep#1}[1]{%
 \def\rep@title{#2 \ref{##1}}%
 \begin{rep@theorem}}%
 {\end{rep@theorem}}}
\makeatother
% ...

\theoremstyle{plain}
\newtheorem{thm}{Theorem}
\newtheorem{lem}[thm]{Lemma}
\newtheorem{prop}[thm]{Proposition}
\newtheorem{cor}[thm]{Corollary}
\newtheorem{claim}{Claim}
\newtheorem*{claim*}{Claim}
\newreptheorem{thm}{Theorem}
\newreptheorem{lem}{Lemma}
\newreptheorem{prop}{Proposition}

\theoremstyle{definition}
\newtheorem{defn}[thm]{Definition}
\newtheorem*{definition*}{Definition}
\newtheorem{conj}[thm]{Conjecture}
\newreptheorem{conj}{Conjecture}
\newreptheorem{defn}{Definition}

\newtheorem*{example}{Example}

\theoremstyle{remark}

\newtheorem*{remark}{Remark}

\usepackage{fancyhdr}
\begin{document}

\title{Upper and lower bounds for rich lines in grids}
\author{Brendan Murphy}
\date{\today} 
\maketitle

\begin{abstract}
We prove upper and lower bounds for the number of lines in general position that are rich in a Cartesian product point set.
This disproves a conjecture of Solymosi and improves work of Elekes, Borenstein and Croot, and Amirkhanyan, Bush, Croot, and Pryby.

The upper bounds are based on a version of the asymmetric \bsg{} theorem for \emph{group actions} combined with product theorems for the affine group.
The lower bounds are based on a connection between rich lines in Cartesian product sets and \emph{amenability} (or expanding families of graphs in the finite field case).

As an application of our upper bounds for rich lines in grids, we give a geometric proof of the asymmetric sum-product estimates of Bourgain and Shkredov.
\end{abstract}

\setcounter{tocdepth}{1}
\tableofcontents

\section{Introduction}

Let $\F$ be a field and let $0<\alpha\leq 1$ be a real number.

A line $\ell$ in the plane $\F^2$ is \emph{$\alpha$-rich}\/ in a Cartesian product point set $Y\times Y\subseteq\F^2$ if
\[
|\ell\cap (Y\times Y)|\geq \alpha|Y|.
\]
For short, we call $Y\times Y$ a \emph{grid}.
Any line contains at most $N$ points of a $N\times N$ grid, so a line is $\alpha$-rich if it contains $\alpha$-percent of the maximum possible points of incidence.
The parameter $\alpha$ may be a constant independent of $N$, or may be some small power of $1/N$.

There are two questions we wish to answer about rich lines in grids:
\begin{enumerate}
\item how many $\alpha$-rich lines can a $N\times N$ grid support?
\item if a grid supports many rich lines, must these lines have some \emph{structure}?
\end{enumerate}
The first question was answered for $\F=\R$ by \ST{} \cite{szemeredi1983extremal}: a $N\times N$ grid has at most $O(\alpha^{-3}N)$ $\alpha$-rich lines; this is sharp.
Below, we discuss extensions of this upper bound and the examples that provide lower bounds.
(We use standard asymptotic notation; see Section~\ref{sec:notation} for definitions.)

The second question is an \emph{inverse problem}\/ for point-line incidences.
The inverse problem for the \ST{} theorem is to show that if $n$ points and $n$ lines in $\R^2$ have $\Omega(n^{4/3})$ incidences, then the point set has some \emph{structure} \cite[Problem 5.7]{croot2007open}; sharpness examples suggest that the point set might contain a large Cartesian product of arithmetic progressions.
Even under the assumption that the point set is a Cartesian product, little is known about the inverse problem for \ST{}.
Question 2 has the further simplifying assumption that the lines are \emph{rich}; in this case, it is possible to give a precise description of the set of lines and the point set \cite{elekes1998linear,elekes2002versus}.

Solymosi conjectured that in the \emph{absence}\/ of structure, a grid can support at most a \emph{constant}\/ number of $\alpha$-rich lines \cite[Conjecture 3.10]{elekes2002versus}.
A generic collection of lines contains no two parallel lines and no three lines through a common point; such a set of lines is said to be in \emph{general position}.
Solymosi's conjecture is the following.
\begin{conj}[{\cite[Conjecture 3.10]{elekes2002versus}}]
  \label{conj:solymosi}
Among the lines that are $\alpha$-rich in a $N\times N$ Cartesian product, at most $C=C(\alpha)>0$ can be in general position.
\end{conj}
In \cite{elekes2002versus}, Conjecture~\ref{conj:solymosi} is stated for lines defined over $\R$ or $\C$.
Solymosi's conjecture is supported by the sharpness examples for the \ST{} incidence bound, and also implies a plausible conjecture of Elekes \cite[Problem 3.9]{elekes2002versus}; see Section 9 of \cite{amirkhanyan2017lines} for a discussion.

Despite this evidence, Conjecture~\ref{conj:solymosi} is false: we disprove it with explicit examples over $\Q,\C,$ and $\F_p$, the finite field with prime cardinality $p$.
The examples we give are quite different from Cartesian products of arithmetic (or geometric) progressions, which show that the \ST{} incidence bound is sharp and motivate the sum-product conjecture.

Let $RLGP(\F, N,\alpha)$ denote the maximum over all $Y\subseteq\F$ with $|Y|=N$ of the maximum number of lines that are $\alpha$-rich in  $Y\times Y$ and in general position.
Explicitly, if $RLGP(Y,\alpha)$ is the maximum number of $\alpha$-rich lines in $Y\times Y$ that are in general position, then
\[
RLGP(\F,N,\alpha):= \max_{Y\subseteq\F, |Y|=N} RLGP(Y,\alpha).
\]
Conjecture~\ref{conj:solymosi} posits that for $\F=\R$ (or $\F=\C$) and for all $0<\alpha<1$, there is a constant $C(\alpha)>0$ depending on $\alpha$ such that
\[
RLGP(\F,N,\alpha)\leq C(\alpha).
\]

For $\F=\Q$, we prove a lower bound for $RLGP(\Q,N,\alpha)$ that is nearly logarithmic in $N$.
In particular, this disproves Conjecture~\ref{conj:solymosi} for $\F=\R$ and $\F=\C$.
\begin{thm}
  \label{thm:20}
There is an absolute constant $C>0$ such that for any $0<\alpha<1$
\[
RLGP(\Q,N,\alpha) \geq C (1-\alpha) \frac{\log N}{\log\log N}.
\]
\end{thm}
For $\F=\C$ and $\F=\F_p$, we prove upper and lower bounds for $RLGP(\F,N,\alpha)$ whose logarithms differ by a square root.
\begin{thm}
  \label{thm:21}
Let $\F$ denote $\C$ or $\F_p$.
For every $0<\alpha<1$, there is a constant $C_\alpha>0$ such that
\[
\frac 1{C_\alpha}\sqrt{\frac{\log N}{\log\log N}} \leq \log \, RLGP(\F,N,\alpha) \leq C_\alpha\frac{\log N}{\log\log N}.
\]
If $\F=\F_p$, the upper bound holds only if $N^{1+\log(2/\alpha)/\log\log N}\leq p$.
\end{thm}
The upper bound in Theorem~\ref{thm:21} 
applies when $\F=\R$, since we may consider points and lines defined over $\R$ to be contained in $\C^2$.

The upper bound in Theorem~\ref{thm:21} is a special case of the following general structure theorem for rich lines in grids over $\C$ and $\F_p$.
\begin{thm}
  \label{thm:4}
There is an absolute constant $C>0$ such that the following holds.
Let $Y$ be a finite subset of $\F$ and let $L$ be a set of $\alpha$-rich lines in $Y\times Y$.
Let $J>0$ be an integer such that $(\alpha/2)^{2^J}\geq 1/|Y|$.

If $\F=\C$, then there is a subset $L'\subseteq L$ such that
\begin{enumerate}
\item the lines of $L'$ are either parallel or concurrent, and
\item $|L'|\gg \left( \frac\alpha2 \right)^{C\, 2^J} |Y|^{-C/J}|L|$.
\end{enumerate}
If $\F=\F_p$, then the same conclusion holds, provided that $|Y| \leq (\alpha/2)^{2^J}p$.
\end{thm}
The key point is that by taking $J$ sufficiently large, the factor $|Y|^{-C/J}$ becomes negligible.
Theorem~\ref{thm:4} is a consequence of a version of the asymmetric \bsg{} theorem for \emph{group actions}, proved in \cite{murphy2017group}, combined with a \emph{product theorem}\/ for the affine group.

The lower bounds in Theorems~\ref{thm:20} and \ref{thm:21} follow from explicit constructions; see Theorems~\ref{thm:28}, \ref{thm:15}, and \ref{thm:24} in Section~\ref{sec:lower-bounds-rich}.
If Conjecture~\ref{conj:solymosi} \emph{were}\/ true over $\R$, then subgroups of $\aff 1\R$ generated by a finite set of affine transformations in general position would not be \emph{amenable}, however finitely generated solvable groups are amenable; Theorem~\ref{thm:28} proves this quantitatively.
Heuristically, if Conjecture~\ref{conj:solymosi} were true over $\F_p$, it might be possible to make an expanding family of Schreier graphs for $\aff 1{\F_p}\actson\F_p$ with bounded degree (following a similar strategy to Bourgain and Gamburd \cite{bourgain2008uniform}), however this is known to be false by a theorem of Lubotzky and Weiss \cite{lubotzky1993groups,lubotzky1994discrete}.
To prove the lower bound in Theorem~\ref{thm:21}, we use a construction of Klawe \cite{klawe1981non-existence,klawe1984limitations}, which gives a quantitative proof of Lubotzky and Weiss' theorem for $\aff 1{\F_p}\actson\F_p$; using a theorem of Grosu~\cite{grosu2014locally}, we embed our counter-example into $\C^2$.

In Section~\ref{sec:lower-bounds-rich} we construct examples of grids that support many $\alpha$-rich lines, and in Section~\ref{sec:upper-bounds-rich} we prove upper bounds the number of $\alpha$-rich lines supported by a $N\times N$ grid.
These sections are completely independent.
The remainder of the introduction contains background on rich lines in grids and some positive results towards Conjecture~\ref{conj:solymosi}, as well as an explanation of the connection between rich lines and grids and sum-product problems.

For completeness, we sketch the proof of the group action version of the \bsg{} theorem and prove the necessary product theorems for affine transformations in Appendices~\ref{sec:proof-group-action-bsg} and \ref{sec:prod-theor-affine}.
In particular, Appendix~\ref{sec:proof-group-action-bsg} gives a proof of Elekes' Theorem~\ref{thm:22} and compares it with the proof of Theorem~\ref{thm:4}.

\subsection{Background on rich lines in grids}

As mentioned, the \ST{} theorem \cite{szemeredi1983extremal} implies that $O(\alpha^{-3}N)$ lines may be $\alpha$-rich in a $N\times N$ grid in $\R^2$.
This lower bound is attained by two simple examples, up to factors of $\alpha$.
\begin{enumerate}
\item If $Y=\{1,\ldots, N\}$, then the parallel lines $\ell(x)=x+b$ are $\alpha$-rich for $b\ll (1-\alpha)N$, thus $Y\times Y$ supports roughly $N$ parallel $\alpha$-rich lines. 
\item If $Y=\{1,2,\ldots,2^{N-1}\}$, then the lines $\ell(x)=2^jx$ through the origin are $\alpha$-rich for $j\ll (1-\alpha)N$, thus $Y\times Y$ supports roughly $N$ concurrent $\alpha$-rich lines.
\end{enumerate}
A more elaborate example, due to Erd\H{o}s, achieves the correct power of $\alpha$.
\begin{example}
  Let $N$ be a large positive integer, let $Y=\{n\in\Z\colon |n|\leq N\}$ and let $P=Y\times Y$.
For coprime integers $a<b$, define a set of lines
\[
L_{a,b}=\{ y-j = \frac ab (x-i)\colon 1\leq i\leq b, 1\leq j\leq \frac N2\}.
\]
Each line in $L_{a,b}$ is incident to at least $\frac N{2b}-1$ points of $P$, and thus is $\alpha$-rich in $P$ for $b\leq\floor{\frac 1{3\alpha}}$.
On the other hand, the number of such lines is $\Theta(\alpha^{-3}N)$.
See \cite{sheffer2014incidences:} for details.
\end{example}

The following theorem of Elekes~\cite{elekes1997linear,elekes2002versus} says that combinations of examples (1) and (2) are essentially the only possibilities.
\begin{thm}[Elekes]
  \label{thm:22}
Let $0<\alpha\leq 1$ be a constant.
If $N$ lines are $\alpha$-rich in an $N\times N$ grid in $\R^2$, then either
\begin{enumerate}
\item $C\alpha^C N$ lines are parallel, or
\item $C\alpha^C N$ lines are concurrent (incident to a common point),
\end{enumerate}
where $C>0$ is a constant independent of $\alpha$ and $N$.
\end{thm}
By applying Freiman's theorem, Elekes concludes that the family of parallel lines obtained in Theorem~\ref{thm:22} have $y$-intercepts in a generalized arithmetic progression (similarly, if the lines are concurrent, then their slopes are in a generalized geometric progression) \cite{elekes1998linear,elekes2002versus}.

Elekes reduces the proof of Theorem~\ref{thm:22} to a \emph{product theorem}.
If  $A$ and $B$ are finite sets of real affine transformations, then we define their \emph{composition set}\/ by
\[
A\circ B := \{\ell_a\circ\ell_b\colon \ell_a\in A,\ell_b\in B\}.
\]
The collection of affine transformations is a group with product given by composition of functions, so $A\circ B$ is the just the product set of $A$ and $B$.
\begin{thm}[Elekes {\cite[Theorem 1]{elekes1997linear}}]
  \label{thm:6}
For every $K>0$ there is a constant $\rho=\rho(K)>0$ depending on $K$ with the following property.

Suppose $A,B$ are finite sets of real affine transformations with $|A|,|B|\geq N$ and
\[
|A\circ B|\leq KN
\]
Then there exist subsets $A'\subseteq A$ and $B'\subseteq B$ with $|A'|,|B'|\geq \rho N$ such that either
\begin{enumerate}
\item both $A'$ and $B'$ consist of parallel lines, or
\item both $A'$ and $B'$ consist of concurrent lines.
\end{enumerate}
\end{thm}
Though it is not explicit in Elekes' work, $\rho$ depends polynomially on $K$.
Parallel and concurrent lines correspond to cosets of abelian subgroups of the affine group, thus Theorem~\ref{thm:6} is perhaps the first instance of a product theorem for a non-commutative group.
Such theorems have now been studied extensively \cite{breuillard2014brief,breuillard2011approximate, breuillard2012structure,gill2014growth,helfgott2008growth,helfgott2015growth,pyber2014growth}.


The assumption that $|L|\approx |Y|$ is essential for Elekes' reduction of Theorem~\ref{thm:22} to Theorem~\ref{thm:6}.
Borenstein and Croot \cite{borenstein2010lines} made the first step towards removing this restriction.
Building on \cite{borenstein2010lines}, Amirkhanyan, Bush, Croot, and Pryby~\cite{amirkhanyan2017lines} proved an analog of Conjecture~\ref{conj:solymosi} where $\alpha=N^{-\delta}$ for some small $\delta>0$.
\begin{thm}[Amirkhanyan, Bush, Croot, and Pryby]
  \label{thm:23}
For all $\epsilon>0$ there exists a $\delta>0$ such that the following holds for all sufficiently large positive integers $N$:

If $L$ is a set of $N^\epsilon$ lines in $\R^2$ that are $\alpha=N^{-\delta}$-rich in an $N\times N$ grid, then the lines of $L$ are not in general position.
\end{thm}

Theorem~\ref{thm:23} implies that for all $\epsilon>0$, if $N$ is sufficiently large, then
\[
RLGP(\R,N,\alpha)\leq N^\epsilon.
\]
In \cite{amirkhanyan2017lines,borenstein2010lines}, the relationship between $\epsilon$ and $\delta$ is not explicit, so it is unclear how strong of a bound this method can achieve.

Borenstein and Croot roughly follow Elekes' method: they reduce to the case of small product set, then contradict structural hypotheses about the initial set of lines.
They do not use Theorem~\ref{thm:6} (or a similar theorem), but instead use sum-product results, some of which are unique to $\R$.
In particular, it is not clear that their methods should extend to $\F_p$ or to other questions about rich transformations for other groups, such as linear fractional transformations \cite{elekes2001combinatorics}.

We use a \emph{group action}\/ version of the \bsg{} theorem \cite{murphy2017group} to reduce the proof of Theorem~\ref{thm:4} to a product theorem for the affine group; in particular, over $\R$ we could use Elekes' Theorem~\ref{thm:6}.
The group action \bsg{} theorem is a generalization of Tao and Vu's asymmetric \bsg{} theorem \cite[Theorem 2.35]{tao2010additive}.
Helfgott pointed out that Borenstein and Croot's method is similar to Tao and Vu's method \cite{borenstein2010lines}.
The group action \bsg{} theorem is a common generalization of these methods.

\subsection{Connection to the sum-product problem}

Theorem~\ref{thm:22} implies a non-trivial \emph{sum-product estimate}.
A sum-product estimate is a lower bound of the form
\[
|A+A|+|AA| \gg |A|^{1+c}
\]
where $A$ is a finite subset of $\R$ (or more generally, a ring), $c>0$, and $A+A$ and $AA$ are the sets of pairwise sums and pairwise products of elements in $A$, respectively.
Erd\H{o}s and Szemer\'edi~\cite{erdos1983products} conjectured that $c$ can be taken arbitrarily close to $1$.
Elekes~\cite{elekes1997number} gave a beautiful geometric proof of a sum-product estimate with $c=1/4$, based on the \ST{} bound.

The following sum-product estimate follows from Theorem~\ref{thm:22}, using the method of \cite{elekes1997number}.
\begin{cor}
  \label{cor:4}
Let $A,B,$ and $C$ be finite subsets of $\R$ with $|B||C|= |A|$.
There is an absolute constant $c>0$ such that
if $|B|, |C|\geq |A|^\epsilon$ for some $\epsilon>0$, then 
\[
|A+B|+|AC|\gg |A|^{1+c\,\epsilon}.
\]
\end{cor}
\begin{proof}
  Let $\ell_{b,c}(x)=c(x-b)$ and let $L$ denote the set of $\ell_{b,c}$ with $b\in B$ and $c\in C$.
Since $|B|,|C|\geq |A|^{\epsilon}$, at most $|A|^{1-\epsilon}$ lines of $L$ are parallel or concurrent.

Set $Y=(A+B)\cup (AC)$.
Each line of $L$ is incident to at least $|A|$ points of $Y\times Y$.

If
\[
|A+B|+|AC|\leq K|A|,
\]
then each line of $L$ is $\alpha$-rich in $Y\times Y$ with $\alpha = 1/2K$.

By Theorem~\ref{thm:22}, at least $C_0\alpha^{C_0}|L|$ lines of $L$ are parallel or concurrent, thus we have
\[
K^{-C_0}|A| \ll |A|^{1-\epsilon} \implies K\gg |A|^{\epsilon/C_0}.
\]
By choosing $K$ to be a sufficiently small power of $|A|$, we have a contradiction.
\end{proof}

The stronger conclusion of Theorem~\ref{thm:4} over Theorem~\ref{thm:23} allows us to give a geometric proof of Bourgain's asymmetric sum-product estimate  \cite{bourgain2005sum-product}.
\begin{thm}[Asymmetric sum-product estimate]
  \label{thm:29}
Let $A,B,$ and $C$ be finite subsets of a field $\F$.

If $\F=\C$ and there is an $\epsilon > 0$ such that $|B|, |C|\geq |A|^{\epsilon}$, then there exists a constant $c=c(\epsilon)>0$ such that
\[
|A+B|+|AC|\gg |A|^{1+c}.
\]
If $\F=\F_p$, the same result holds provided that $|A| \ll p^{1-O(\epsilon)}$.
\end{thm}
In fact, we achieve estimates comparable to those of Shkredov \cite{shkredov2017remarks}: we may take $c=1/(J2^J)$ for $J\approx\gamma\epsilon$.
\begin{thm}[Asymmetric sum-product theorem]
  \label{thm:2}
Suppose that $A,B,C\subseteq\F$ are finite.
Let $J>0$ be a positive integer and let $0< K \leq \frac 12 |A|^{1/2^J}$ be a parameter.

If $\F=\C$, then either
\begin{equation}
  \label{eq:5}
  |AC|+|A+B| > K|A|,
\end{equation}
or
\begin{equation}
  \label{eq:6}
  \min(|B|,|C|) \ll K^{C 2^J} |A|^{C/J}.
\end{equation}

If $\F=\F_p$, the same dichotomy holds, provided that $|A|\leq (2K)^{-2^J} p$.
\end{thm}

\subsection{Acknowledgements}

I would like to thank the following people for helpful conversations and for reading various drafts of this work:
Ernie Croot,  
Harald Helfgott,
Alex Iosevich,
John Mackay,  
Jonathan Pakianathan,
Sarah Peluse,
Giorgis Petridis,
Misha Rudnev,  
Ilya Shkredov,  
Jozsef Solymosi,
Sophie Stevens, and 
Yufei Zhao.
I would also like to thank the Arizona Winter School and the Simons Institute, whose workshops resulted in several of the keys ideas in this work, as well as the Heilbronn Institute for Mathematical Research and the University of Rochester, who funded these trips.

\subsection{Notation}
\label{sec:notation}

We use standard asymptotic notation: $f=O(g)$ means that there is a constant $C>0$ such that $|f(x)|\leq C g(x)$ for all $x$; $f\ll g$ means the same as $f=O(g)$, $f=\Omega(g)$ and $f\gg g$ mean the same as $g\ll f$.
The notation $f\approx g$ means that $f\ll g$ and $g\ll f$; $f=\Theta(g)$ means $f\approx g$.
We abuse asymptotic notation slightly for stating hypotheses: a condition of the form $f\ll g$ means that there exists a constant $C$ such that if $|f|\leq Cg$, then the theorem holds.
Notation such as $f\ll_\alpha g$ or $f=O_\epsilon(g)$ means that the implicit constant $C$ depends on the parameter in the subscript.

Unless otherwise stated, we use the following notation throughout:
\begin{itemize}
\item  $\alpha$ denotes a real number in $(0,1]$,
\item lower case Greek letters denote (typically small) real parameters,
\item $C$ denotes a positive constant, which may change from line to line, 
\item  $\F$ denotes a field, which may be $\R,\C,\Q,$ or $\F_p$, the finite field with prime cardinality $p$,
\item $Y$ denotes a finite subset of $\F$, and $N$ denotes $|Y|$,
\item $L$ denotes a finite set of lines in $\F^2=\F\times\F$,
\item $G$ denotes the group $\aff 1\F$ of affine transformations of $\F$; we represent elements of $\aff 1\F$ by linear functions $x\mapsto ax+b$ with $a,b\in F$, $a\not=0$, with composition as the group operation,
\item $A$ denotes a finite subset of $G=\aff 1\F$.
\item for $Y\subseteq\F$ and $0<\alpha\leq 1$, we use $\sym_\alpha(Y)$ to denote the set of $g\in G$ such that $|Y\cap gY|\geq\alpha|Y|$; this is called a \emph{symmetry set}\/ of $Y$.
\end{itemize}

\section{Lower bounds for rich lines in grids}
\label{sec:lower-bounds-rich}

In this section, we disprove Conjecture~\ref{conj:solymosi}, which we recall here.
\begin{repconj}{conj:solymosi}
  Among the lines in $\F^2$ that are $\alpha$-rich in an $N\times N$ Cartesian product set, at most $C=C(\alpha)>0$ lines are in general position.
\end{repconj}

In Section~\ref{sec:quant-lower-bounds-R}, we disprove Conjecture~\ref{conj:solymosi} over $\Q$ with an explicit construction.
In Section~\ref{sec:quant-lower-bounds-Fp}, we give an explicit construction of a large set of lines in general position in $\F_p^2$.
In Section~\ref{sec:quant-lower-bounds-C}, we embed the counter-examples from the previous section into $\C^2$.

\subsection{Qualitative lower bound based on amenability}

\begin{thm}
  \label{thm:444414}
Let $\F$ be an infinite field, and let $L$ be an arbitrary set of lines in $\F^2$.
For all $\epsilon >0$ and all $N > 0$, there exists a subset $Y\subseteq \F$ such that $|Y|\geq N$ and $|(Y\times Y)\cap \ell|\geq (1-\epsilon)|Y|$ for all $\ell$ in $L$.
\end{thm}
This shows that \emph{any}\/ set of lines, regardless of their structure, are $(1-\epsilon)$-rich in some grid, provided that the number of lines is sufficiently small compared to the size of the grid.
This result is trivial if $\F$ is finite, since we may take $Y=\F$.

The proof of Theorem~\ref{thm:444414} uses a strategy of Lubotzsky and Weiss \cite{lubotzky1993groups} based on the \emph{amenability}\/ of finitely generate solvable groups, such as finitely generated subsets of $\aff 1\R$.

We recall some basic facts about amenable groups.
\begin{defn}[Amenability]
\label{def:44443}
  Let $\Gamma$ be a group generated by $S=\{\gamma_1,\ldots,\gamma_k\}$ and suppose that $\Gamma$ acts on a set $X$.
We say that $\Gamma$ is \emph{amenable}\/ if there exists a sequence of finite subsets $F_n\subseteq X$ such that for all $\gamma\in S$
\begin{equation}
  \label{eq:444422}
  \lim_{n\to\infty}\frac{|\gamma F_n\Delta F_n|}{|F_n|} = 0.
\end{equation}
\end{defn}
Here $A\Delta B = (A\setminus B)\cup (B\setminus A)$ is the \emph{symmetric difference}.
A sequence of subsets $\{F_n\}$ satisfying (\ref{eq:444422}) is called a \emph{F\o{}lner sequence}.
If $X$ is infinite, then $|F_n|\to\infty$ as $n\to\infty$ (see \cite[Top of p.23]{bartholdi2018amenability}).
See \cite[Lemma 3.6, Lemma 3.7, Theorem 3.23]{bartholdi2018amenability}.

Any finite group $G$ is amenable and any finitely generated abelian group is amenable.
In fact, any group of sub-exponential growth is amenable, as can be seen by taking $\{F_n\}$ to be a sufficiently sparse sequence of balls about the identity in the word metric.
Though solvable groups are not necessarily of sub-exponential growth, finitely generated solvable groups are amenable.
\begin{prop}
  \label{prop:44443}
Finitely generated solvable groups are amenable.
\end{prop}
This is because the property of amenability is preserved by taking extensions and solvable groups can be constructed by extensions by abelian groups \cite[Proposition 7.1]{bartholdi2018amenability}.

We now have the background needed to prove Theorem~\ref{thm:444414}.
\begin{proof}[Proof of Theorem~\ref{thm:444414}]
  Let $S\subseteq\aff 1\F$ be a set of affine transformations corresponding to the lines in $L$ (that is, each $\ell$ in $L$ has equation $y=\gamma(x)$ for some $g\in S$).
Fix a positive integer $N$.
We will show that there is a subset $Y\subseteq\F$ such that $|Y|\geq N$ and $|\gamma(Y)\cap Y|\geq (1-\epsilon)|Y|$ for all $g\in S$.

Let $\Gamma$ denote the subgroup of $\aff 1\F$ generated by $S$.
Since $\Gamma$ is solvable and finitely generated, it is amenable by Proposition~\ref{prop:44443}.
By Definition~\ref{def:44443} there is a F\o{}lner sequence $\{F_n\}$ of subsets of $X$.

By (\ref{eq:444422}), there is a positive integer $n_0$ such that for all $n\geq n_0$ and all $\gamma$ in $S$, we have
\[
|\gamma F_n\Delta F_n| \leq 2\epsilon|F_n|.
\]
Since
\[
|\gamma F_n\Delta F_n| = 2(|F_n|-|F_n\cap\gamma F_n|),
\]
we have $|\gamma(F_n)\cap F_n|\geq (1-\epsilon)|F_n|$ for all $n\geq n_0$.

Since $X=\F$ is infinite, $|F_n|\to\infty$ as $n\to\infty$.
It follows that for some $n_1\geq n_0$, if $Y=F_{n_1}$ then $|Y|\geq N$.
\end{proof}
Lubotzky~\cite[Proposition 3.3.6]{lubotzky1994discrete} has shown how to apply this strategy to $\aff 1{\F_p}$ acting on $\F_p$, which would give a qualitative theorem of the same sort for sequences of finite affine groups.
Rather than take this approach, we give an explicit example in Section~\ref{sec:quant-lower-bounds-Fp}, based on work of Klawe \cite{klawe1981non-existence}.

\subsection{Quantitative lower bounds over $\R$}
\label{sec:quant-lower-bounds-R}

In this section, we give and explicit construction of arbitrarily large finite sets $Y$ in $\R$ such that 
$Y\times Y$ supports
a large number of affine transformations in general position.
In fact, the construction is defined over the integers.
\begin{thm}
  \label{thm:28}
 For all $0<\alpha\leq 1$ and all $N_0 >0$ there exists a set $Y\subseteq\Z$ such that $|Y|\geq N_0$ and $Y\times Y$ supports a set $L$ of $\alpha$-rich lines in general position such that
\[
|L| \gg (1-\alpha) \frac{\log |Y|}{\log\log|Y|}.
\]
\end{thm}
The construction is based on the construction of explicit F\o{}lner sequences for $\aff 1\R$ acting on $\R$ \cite{greenleaf1969invariant,willson2006folner}.
We thank John Mackay for suggesting a simpler way of writing our original set $Y$.
\begin{proof}[Proof of Theorem~\ref{thm:28}]
Fix $0<\epsilon<1$ such that $2\epsilon \leq 1-\alpha$.
Fix an integer $N > 0$ so that $N^{N+1} \geq N_0$.

Define a set of positive integers $Y\subseteq\Z$ of size $N^{N+1}$ by
\[
Y:=\bigcup_{k=0}^{N-1} N^k\cdot \left( N^N + [0,N^N)\cap \Z \right).
\]
Since $|[0,N^N)\cap \Z|=N^N$ and the terms of the union are disjoint, we have $|Y|=N^{N+1}$.

Let $L$ denote the set of transformations defined by $\ell_k(x):=N^kx+k N^{N-1}$, where $k$ ranges over integers satisfying $0<k<\epsilon N$.
We make two claims.
\begin{claim}
\label{claim:1}
  The transformations in $L$ are in general position.
\end{claim}
\begin{claim}
  \label{claim:2}
For all $\ell$ in $L$, we have
\[
|\ell(Y)\setminus Y| \leq 2\epsilon |Y|,
\]
hence $|\ell\cap (Y\times Y)|\geq \alpha|Y|$ by our choice of $\epsilon$.
\end{claim}
The proof is complete assuming these claims, since $|L|\approx \epsilon N \approx \epsilon \log|Y|/\log\log |Y|$.

To prove Claim~\ref{claim:1}, it suffices to show that if $0<i<j<k<\epsilon N$, then
\begin{equation}
  \label{eq:68}
  \det
\begin{pmatrix}
  1 & 1 & 1 \\
  N^i & N^j & N^k \\
  i & j & k\\
\end{pmatrix}
\not=0.
\end{equation}
(We have factored the common term $N^{N-1}$ out of the bottom row.)
The left-hand side of (\ref{eq:68}) is
\[
(k-i)N^j - (j-i)N^k - (k-j)N^i < (k-i)N^j - (j-i)N^k,
\]
which is strictly less than zero:
\[
(k-i)N^j < \epsilon N^{j+1} \leq \epsilon N^k \leq N^k \leq (j-i)N^k.
\]

To prove Claim~\ref{claim:2}, fix an element $\ell_b$ in $L$ and consider its action on a general element $y=N^k(N^N+x)$ of $Y$, where $x$ is an integer in $[0,N^N)$:
\[
\ell_b(y)= N^{b+k}(N^N + x) + bN^{N-1} = N^{b+k}(N^N+x + bN^{N-1-b-k}).
\]
There are two cases where $\ell_b(y)\not\in Y$:
\begin{enumerate}
\item $b+k \geq N$,
\item $b+k < N$, but $x+N^{N-1-b-k}b\geq N$.
\end{enumerate}
At most $b$ values of $k$ satisfy (1), hence at most $bN^{N}\leq \epsilon |Y|$ elements of $Y$ fall into the first case.
At most $N^{N-1-b-k}b\leq \epsilon N^N$ values of $x$ satisfy (2), hence at most $N\cdot \epsilon N^{N}\leq \epsilon|Y|$ elements of $Y$ fall into the second case.
\end{proof}

\subsection{Quantitative lower bounds over $\F_p$}
\label{sec:quant-lower-bounds-Fp}

In this section, we prove the lower bound in Theorem~\ref{thm:21} for $\F=\F_p$.
\begin{thm}
  \label{thm:15}
For any prime $p$, any $0<\alpha<1$, any $\epsilon >0$, and any integer $m$ satisfying $1\ll_\epsilon m \ll p^{1-\epsilon}$, there exists a subset $Y\subseteq\F_p$ with $|Y|\approx_\alpha m$ and a set of lines $S$ in general position that are $\alpha$-rich in $Y\times Y$ such that
\[
\log |S| \approx_\alpha \sqrt{\frac{\log |Y|}{\log\log |Y|}}.
\]
\end{thm}

The proof of Theorem~\ref{thm:15} is based on a construction of Klawe~\cite{klawe1981non-existence, klawe1984limitations}, which proves explicitly that Schreier graphs of $\aff 1{\Z/n\Z}\actson \Z/n\Z$ cannot be made into an expander family of constant degree.
(Lubotzky \cite{lubotzky1994discrete} gives a qualitative proof of this fact using the method of \cite{lubotzky1993groups}.)

Before we state Klawe's theorem, we need some notation.
Let $Q=\{q_1,\ldots,q_k\}$ denote a set of $k$ primes.
We say that $n$ is a \emph{$Q$-power}\/ if $n=q_1^{\alpha_1}\cdots q_k^{\alpha_k}$ and in this case, we write $\mu(n)=\alpha_1+\cdots+\alpha_k$.
We use $\phi(n)$ to denote Euler's totient function; that is, $\phi(n)$ the number of positive integers less than and relatively prime to $n$.
\begin{thm}[Klawe]
  \label{thm:16}
Let $Q=\{q_1,\ldots,q_k\}$ be a set of $k$ prime numbers and set $q=q_1\cdots q_k$.
Let $N,M,L,r,$ and $s$ be positive integers such that $N=Mq^s+r$, $0\leq r < q,$ and $L<M/q^s$.

Then there exists a subset $Y\subseteq \Z/N\Z$ such that
\begin{equation}
  \label{eq:23}
  |Y|=s^k L\phi(q)q^{s-1}
\end{equation}
and for all positive integers $0< a,b<N$ such that $a$ is a $Q$-power
\begin{equation}
  \label{eq:24}
  |(aY+b)\setminus Y| \leq \left( \frac{\mu(a)}s+\frac{ar+b}L\,\frac{q}{\phi(q)} \right)|Y|.
\end{equation}
\end{thm}
The proof of Theorem~\ref{thm:16} uses a construction similar to that of Theorem~\ref{thm:28}, but uses wrap-around to allow a much larger set of ``slopes'' $a$.

We use the following corollary of Theorem~\ref{thm:16} to prove Theorem~\ref{thm:15}.
\begin{cor}
  \label{cor:2}
Let $Q=\{q_1,\ldots,q_k\}$ be a set of $k$ prime numbers and set $q=q_1\cdots q_k$.
Let $p$ be a prime and let $M,L,r,$ and $s$ be positive integers such that $p=Mq^s+r$, $0\leq r < q$, and $L<M/q^s$.

If
\begin{equation}
  \label{eq:25}
  L \geq\frac{8q}{\phi(q)}\,\max \left( q^{(1+\frac 1{4k})s}, \left( \frac s{4k} \right)^{2k} \right),
\end{equation}
then there exists a subset $Y\subseteq\F_p$ satisfying (\ref{eq:23}) and a set $S$ of affine transformations in general position such that $|S|\geq (s/4k)^k$ and $|\ell(Y)\setminus Y|\leq\frac 12|Y|$ for all transformations $\ell$ in $S$.
\end{cor}
\begin{proof}
  Applying Theorem~\ref{thm:16} with $N=p$ yields a set $Y\subseteq\F_p$ such that (\ref{eq:23}) holds and for all positive integers $a$ and $b$ such that $a$ is a $Q$-power (\ref{eq:24}) holds.

We wish to choose a collection of pairs $(a,b)$ such that the corresponding set of lines $\ell(x)=ax+b$ are in general position and satisfy
\begin{equation}
  \label{eq:44}
  \frac{\mu(a)}s+\frac{ar+b}L\,\frac{q}{\phi(q)} \leq \frac 12.
\end{equation}
First, we will find a large number of integers $a,b$ satisfying
\begin{equation}
  \label{eq:45}
\mu(a)\leq \frac s4
\end{equation}
and
\begin{equation}
  \label{eq:46}
  ar+b \leq \frac{\phi(q)L}{4q}.
\end{equation}

Let $A$ denote the set of positive integers of the form $a=q_1^{\alpha_1}\cdots q_k^{\alpha_k}$ where $0\leq \alpha_i\leq s/4k$ for $i=1,\ldots,k$.
Then each element $a$ in $A$ satisfies (\ref{eq:45}) and further $1\leq a\leq q^{s/4k}$.

If $a\in A$ and $b$ is a positive integer, then
\[
ar+b < q^{(1+\frac 1{4k})s} + b.
\]
By (\ref{eq:25}), we have $\phi(q)L/4q\geq 2q^{(1+\frac 1{4k})s}$, so (\ref{eq:46}) is satisfied for all $0\leq b\leq \phi(q)L/8q$.

Note that $q^{s/4k}<p$ and $\phi(q)L/8q<p$, so $a$ and $b$ are unique modulo $p$.

To form our set of lines $L$, we will choose slopes $a$ from $A$ one at a time, choosing $y$-intercepts $0\leq b\leq \phi(q)L/8q$ so that the line $\ell(x)=ax+b$ intersects each previous line in a distinct point; this guarantees that no three lines in $L$ are incident to a common point.
Since all of the lines in $L$ have distinct slopes, the resulting set of lines $L$ will be in general position.

If we have chosen $x$ lines by this process, then we must avoid ${x\choose 2}$ points; this is always possible if we have more than ${x\choose 2}$ choices for $b$.
By (\ref{eq:25}),
\[
\#(\mbox{choices for for $b$}) \geq \frac{\phi(q)L}{8q} \geq \left( \frac s{4k} \right)^{2k} > {x\choose 2}\quad\mbox{for all $0\leq x\leq |A|$}.
\]
\end{proof}

\newcommand{\merr}{\left( 1+O\left( \frac 1{\log x} \right) \right)}
We want to take $k$ as large as possible relative to $q$; the following lemma gives $k\approx \log q/\log\log q$.
\begin{lem}
  \label{lem:13}
Given $x>0$, let $Q=\{q_1,\ldots,q_k\}$ denote the set of primes less than or equal to $x$, and let $q=q_1\cdots q_k$.
We have the following estimates:
\begin{equation}
  \label{eq:26}
  k=|Q|=\frac{x}{\log x} \merr,
\end{equation}
\begin{equation}
  \label{eq:27}
  q=e^{x\merr},
\end{equation}
and
\begin{equation}
  \label{eq:28}
  \frac{\phi(q)}q = \frac{e^{-\gamma}}{\log x}\merr,
\end{equation}
where $\gamma$ is Euler's constant.
\end{lem}
\begin{proof}
  Equation~(\ref{eq:26}) is the Prime Number Theorem.
Equation~(\ref{eq:27}) follows from asymptotic estimates for Chebyshev's function $\vartheta(x)$:
\[
\vartheta(x)=\sum_{p\leq x}\log p = x\merr.
\]
Equation~(\ref{eq:28}) is Merten's formula \cite[Equation (2.16)]{iwaniec2004analytic}.
\end{proof}

For simplicity, we will prove Theorem~\ref{thm:15} for the case $\alpha=\frac 12$; the general case follows in the same way, with implicit constants depending on $\alpha$.
\begin{proof}[Proof of Theorem~\ref{thm:15}]
  Let $x$ be a positive real number and let $Q=\{q_1,\ldots,q_k\}$ denote the set of primes less than or equal to $x$.
Let $q=q_1\cdots q_k$ and let $s=\ceil{4e\cdot k}$.
For convenience, let $\delta=1/4k$.

Set
\[
L = \frac{8q}{\phi(q)}q^{(1+\delta)s}.
\]
Condition~(\ref{eq:25}) of Corollary~\ref{cor:2} holds if $q^{(1+\delta)s}\geq (s/4k)^{2k}$.
By Lemma~\ref{lem:13},
\[
q^{(1+\delta)s}\geq q^{4k} = e^{2k\cdot 2x\merr},
\]
while
\[
\left( \frac s{4k} \right)^{2k} \leq \left( e+\frac 1{4k} \right)^{2k} \ll e^{2k},
\]
thus condition~(\ref{eq:25}) holds if $x\gg 1$.

Write $p=Mq^s+r$, where $0\leq r < q^s$ and $M>q^sL$.
By Corollary~\ref{cor:2} there is a set $Y\subseteq\F_p$ such that
\begin{equation}
  \label{eq:47}
  |Y|=s^k L \phi(q)q^{s-1} = 8s^kq^{(2+\delta)s}
\end{equation}
and a set $S$ of lines in general position that are $\frac 12$-rich in $Y\times Y$ such that
\begin{equation}
  \label{eq:48}
  |S|\geq \left( \frac s{4k} \right)^k \geq e^k.
\end{equation}

By Lemma~\ref{lem:13},
\[
\log |S| \geq k \sim\frac{x}{\log x},
\]
while
\[
\log |Y|\approx k\log s + (2+\delta)s\log q \approx \frac{x^2}{\log x}.
\]
Thus
\[
\log |S| \approx \sqrt{\frac{\log |Y|}{\log\log |Y|}},
\]
as desired.

Now we will derive constraints on $m=|Y|$.
Since $x\gg 1$, we have $m\gg 1$.
On the other hand, we must have
\[
p\geq Mq^s \geq q^{2s}L=\frac{8q}{\phi(q)}q^{(3+\delta)s}\sim 8e^{\gamma}\log x q^{(3+\delta)s}\approx (\log\log q) q^{(3+\delta)s}.
\]
Since $k\geq \frac x{\log x} \left( 1 - \frac {C}{\log x} \right)$ and $q\approx e^x$, we have
\[
s^k \gg k^k \gg \left( \frac x{\log x} \right)^{k} \gg \frac{q}{q^{C/\log\log q}}.
\]
Thus
\[
|Y| \gg \frac{q^{(3+\delta)s}}{q^{C/\log\log q}}.
\]
Thus to ensure $(\log\log q)q^{(3+\delta)s} \ll p$, it suffices to take $|Y| \leq p^{1-\epsilon}$ for any $\epsilon > 0$.
\end{proof}

\subsection{Quantitative lower bounds over $\C$}
\label{sec:quant-lower-bounds-C}
In this section, we prove the lower bound in Theorem~\ref{thm:21} for $\F=\C$.
\begin{thm}
  \label{thm:24}
For all $0<\alpha<1$ there exists an absolute constant $N_0\geq 0$ such that for all $N\geq N_0$ there is a subset $Y\subseteq \C$ with that $|Y|\geq N$ and a set $S$ of lines in general position that are $\alpha$-rich in $Y\times Y$ and satisfy
\[
\log |S|\approx_\alpha \sqrt{\frac{\log |Y|}{\log\log |Y|}}.
\]
\end{thm}
The proof of Theorem~\ref{thm:24} is an application of a \emph{rectification theorem}\/ of Grosu~\cite{grosu2014locally}, which allows us to embed small subsets of $\F_p$ into $\C$ while preserving algebraic equations of low complexity.
In particular, Grosu's theorem allows up to embed the counterexamples constructed in Theorem~\ref{thm:15} into $\C^2$.
This seems to be the first time that Grosu's theorem has been used to prove a counterexample to a statement over $\C$, rather than to prove a positive statement for very small subsets of $\F_p$.

Before we state Grosu's theorem, we need some definitions.
A polynomial $f\in\Z[x_1,\ldots,x_n]$ is \emph{$k$-bounded}\/ if $\deg(f)\leq k$ and the sum of the absolute values of the coefficients of $f$ are bounded by $k$.
Given rings $R_1$ and $R_2$ and subsets $A=\{a_1,\ldots,a_n\}\subseteq R_1$ and $B\subseteq R_2$, we call a bijection $\phi\colon A\to B$ a \emph{Freiman ring isomorphism of order $k$}\/ (or $F_k$-ring isomorphism) if for any $k$-bounded $f\in\Z[x_1,\ldots, x_n]$ we have
\[
f(a_1,\ldots,a_n)=0\iff f(\phi(a_1)),\ldots,f(\phi(a_n)))=0.
\]
\begin{thm}[Grosu]
  \label{thm:25}
Let $k\geq 2$ be an integer, let $p$ e a prime, and let $A$ be a subset of $\F_p$.
If $|A|<\log_2\log_{2k}\log_{2k^2}p -1$, then there exists a subset $A'\subseteq \C$, and a homomorphism $\phi_p\colon\Z[A']\to\F_p$ such that $\phi_p$ is an $F_k$-ring homomorphism between $A'$ and $A$.
\end{thm}
Grosu used Theorem~\ref{thm:25} to prove that incidence bounds for points and lines in $\C^2$ can be applied to small sets of points and lines in $\F_p^2$ \cite[Theorem 10]{grosu2014locally}.
We give a variation on this argument that guarantees that lines in general position in $\C^2$ correspond to lines in general position in $\F_p^2$.
\begin{cor}
  \label{cor:3}
Let $p$ be a prime, let $Y$ be a subset of $\F_p$, and let $S$ be a set of lines in $\F_p^2$ in general position that are $\alpha$-rich in $Y\times Y$ for some $0<\alpha<1$.

If $|Y|+2|S|+{|S|\choose 3} <\log_2\log_{14}\log_{98}p-2$, then there exists a subset $Y'\subseteq \C$ and a set of lines $S'$ in $\C^2$ that are in general position and $\alpha$-rich in $Y'\times Y'$.
\end{cor}
\begin{proof}
  We will show that it suffices to construct a $F_7$-ring isomorphism between a certain subset $A\subseteq \F_p$ and some subset $A'\subseteq \C$.

Suppose that the elements of $S$ have the form $\ell_i(x)=a_ix+b_i$.
If $\ell_i,\ell_j,\ell_k$ are distinct lines that intersect in a common point, then the matrix
\[
  \begin{pmatrix}
    a_i & b_i & 1\\
    a_j & b_j & 1\\
    a_k & b_k & 1\\
  \end{pmatrix}
\]
is singular.
By hypothesis, the lines of $S$ are in general position, so the numbers
\begin{equation}
  \label{eq:50}
  d_{ijk}:=\det
  \begin{pmatrix}
    a_i & b_i & 1\\
    a_j & b_j & 1\\
    a_k & b_k & 1\\
  \end{pmatrix}
\end{equation}
are non-zero.

Let $A$ be the union of $Y, \{a_i\}, \{b_i\}, \{d_{ijk}\},$ and $\{0\}$.
Then by hypothesis
\begin{equation}
  \label{eq:51}
  |A| \leq |Y| + 2|S| +{|S|\choose 3} + 1 < \log_2\log_{14}\log_{98}p - 1.
\end{equation}

For each line $\ell_i$, we have at least $\alpha|Y|$ solutions to
\begin{equation}
  \label{eq:49}
y'=a_iy+b  
\end{equation}
with $y,y'$ in $Y$.
This equation is 3-bounded.
The equation (\ref{eq:50}) is 7-bounded.

By (\ref{eq:51}), we may apply Theorem~\ref{thm:25} to $A$ to find a subset $A'\subseteq \C$ and a $F_7$-ring homorphism $\phi_p\colon\Z[A']\to\F_p$ from $A'$ to $A$.

Let $Y'\subseteq \C$ denote the set of elements in $A'$ that map to $Y$ under $\phi_p$ and let $S'$ denote the set of lines defined by $\ell_i'(x)=a'_ix+b'_i$ where $\phi_p(a_i')=a_i$ and $\phi_p(b'_i)=b_i$.
Since $\phi_p$ is a bijection from $A'$ to $A$, we have $|Y'|=|Y|$ and $|S'|=|S|$.

Since $\phi_p$ preserves \ref{eq:50} and (\ref{eq:49}), the lines of $S'$ are in general position (since by bijectivity, no $d_{ijk}$ is mapped to 0), and each line in $S'$ is incident to at least $\alpha|Y'|$ points of $Y'\times Y'$.
\end{proof}

We are now ready to prove Theorem~\ref{thm:24}.
\begin{proof}[Proof of Theorem~\ref{thm:24}]
  Without loss of generality, we may assume that $|S|<N^{1/3}$.
Choose a prime $p$ so that
\begin{equation}
  \label{eq:52}
  5N\leq \log_2\log_{14}\log_{98}p -2.
\end{equation}
Since $N_0\leq N \leq p^{1/2}$, if $N_0$ is sufficiently large (depending on $\alpha$), then by Theorem~\ref{thm:15} there is a subset $Y\subseteq\F_p$ of size $\approx_\alpha N$ and a set $S$ of lines in $\F_p^2$ in general position and $\alpha$-rich in $Y\times Y$ such that
\[
\log |S|\approx_\alpha \sqrt{\frac{\log|Y|}{\log\log|Y|}}.
\]
By (\ref{eq:52}) we have
\[
|Y| + 2|S| + {|S|\choose 3} \leq 5N \leq  \log_2\log_{14}\log_{98}p -2,
\]
so by Corollary~\ref{cor:3} we may embed $Y$ into $\C$ and $S$ into $\C^2$.
\end{proof}

\section{Upper bounds for rich lines in grids}
\label{sec:upper-bounds-rich}

In this section, we prove two upper bounds for the number of rich lines in a $N\times N$ grid in $\F^2$, and an asymmetric sum-product estimate over $\F$, where $\F=\C$ or $\F=\F_p$.
If $\F=\F_p$, we need an additional contraint to rule out trivial counter-examples: the grid $\F_p\times\F_p$ has $p^2$ 1-rich lines, and $\F_p$ does not grow under addition or multiplication.

These theorems are all consequences of Theorem~\ref{thm:4}, which is a general inverse theorem for rich lines in grids. Theorem~\ref{thm:4} is an immediate corollary of Theorem~\ref{thm:9x}, which is an inverse theorem for rich affine transformations.
The difference between Theorems~\ref{thm:4} and \ref{thm:9x} is a matter of language, and we give a dictionary between geometric and algebraic terminology in Section~\ref{sec:geom-dict-proof}.

First we state the upper bound for $\alpha$-rich lines in a $N\times N$ grid where $\alpha=N^{-\delta}$, which generalizes Theorem~\ref{thm:22} to sets of lines of size $N^\epsilon$ for any $\epsilon>0$, as well as to points and lines defined over $\C$ or $\F_p$.
\begin{thm}[Upper bound, polynomial density]
  \label{thm:3}
For all $\epsilon >0$ and $0<\gamma<1$, there is a $\delta>0$ such the following holds for all $N>0$.

Let $L$ be a set of $N^\epsilon$ lines in $\F^2$ that are $N^{-\delta}$-rich in an $N\times N$ grid.
\begin{itemize}
\item If $\F=\C$, then there is a subset $L'\subseteq L$ of size $|L'|\gg |L|^{1-\gamma}$ such that the lines of $L'$ are either parallel or concurrent.
\item If $\F=\F_p$, the same conclusion holds, provided that $N\ll p^{1-O(\gamma\,\epsilon)}$.
\end{itemize}
Further, we may take $\delta=1/(J\,2^{J})$, where $J\approx \gamma\epsilon$.
\end{thm}
Theorem~\ref{thm:3} immediately implies the main theorem of \cite{amirkhanyan2017lines} (Theorem~\ref{thm:23}), since if the lines of $L$ are in general position, then $|L'|\leq 2$, which yields a contradiction for $N$ sufficiently large.

Next we consider  $\alpha$-rich lines in an $N\times N$ grid where $\alpha$ is fixed.
\begin{thm}[Upper bound, constant density]
\label{thm:5}
For all $0<\alpha<1$ there is a constant $C=C(\alpha)>0$ such that the following holds for all $N>0$.

Let $L$ is a set of lines in $\F^2$ that are in general position and are $\alpha$-rich in an $N\times N$ grid.
\begin{itemize}
\item If $\F=\C$, then \(|L| \ll_\alpha N^{C/\log\log N}.\)
\item If $\F=\F_p$, then same conclusion holds, provided that \(N^{1+\log(2/\alpha)/\log\log N}\leq p\).
\end{itemize}
\end{thm}
Theorem~\ref{thm:5} proves the upper bounds stated in Theorem~\ref{thm:21}.

We will prove the following asymmetric sum-product result, which immediately implies Theorem~\ref{thm:29}.
\begin{repthm}{thm:2}
Suppose that $A,B,C\subseteq\F$ are finite.
Let $J>0$ be a positive integer and let $0< K \leq \frac 12 |A|^{1/2^J}$ be a parameter.

If $\F=\C$, then either
\begin{equation}
  |AC|+|A+B| > K|A|,
\end{equation}
or
\begin{equation}
  \min(|B|,|C|) \ll K^{C 2^J} |A|^{C/J}.
\end{equation}

If $\F=\F_p$, the same dichotomy holds, provided that $|A|\leq (2K)^{-2^J} p$.  
\end{repthm}
Choosing $2K = |A|^{\frac 1{J 2^J}}$ proves Theorem~\ref{thm:29} with $\epsilon= 1/J$, since (\ref{eq:6}) cannot hold for this choice of $K$.

Theorems~\ref{thm:3}, \ref{thm:5}, and \ref{thm:2} are special cases of the following general inverse theorem for rich lines in grids, which we stated in the introduction.
\begin{repthm}{thm:4}
There is an absolute constant $C>0$ such that the following holds.
Let $Y$ be a finite subset of $\F$ and let $L$ be a set of $\alpha$-rich lines in $Y\times Y$.
Let $J>0$ be an integer such that $(\alpha/2)^{2^J}\geq 1/|Y|$.

If $\F=\C$, then there is a subset $L'\subseteq L$ such that
\begin{enumerate}
\item the lines of $L'$ are either parallel or concurrent, and
\item $|L'|\gg \left( \frac\alpha2 \right)^{C\, 2^J} |Y|^{-C/J}|L|$.
\end{enumerate}
If $\F=\F_p$, then the same conclusion holds, provided that $|Y| \leq (\alpha/2)^{2^J}p$.
\end{repthm}
In turn, Theorem~\ref{thm:4} is a simple translation of an \emph{algebraic inverse theorem}\/ for rich affine transformations.

We need some notation.
If $Y$ is a finite subset of $\F$, we let $\sym_\alpha(Y)$ denote the set of transformations $g$ in $\aff 1\F$ such that $|Y\cap gY|\geq \alpha|Y|$.
\begin{thm}[Inverse theorem for $\aff 1\F\actson\F$]
  \label{thm:9x}
Let $\F$ denote $\C$ or $\F_p$.
There exists an absolute constant $C>0$ such that the following holds:

Suppose that $Y\subseteq\F$ is finite, $0<\alpha<1$, and $A\subseteq\sym_\alpha(Y)$.

Let $J\geq 0$ be an integer such that $(\alpha/2)^{2^J}\geq 1/|Y|$, and if $\F=\F_p$, suppose that $J$ also satisfies $|Y|\leq \left( \frac\alpha2 \right)^{2^J}p$.

Then there is an element $g$ in $G$ and an \emph{abelian}\/ subgroup $H$ of $G$ such that
\[
|A\cap gH|\gg \left( \frac\alpha2 \right)^{C\,2^J}|Y|^{-C/J}|A|.
\]
\end{thm}
The significance of $A$ containing many elements in an abelian subgroup is that the only way to have many rich transformations is by popular differences or popular ratios.

In the first subsection, we give a dictionary between the geometric language of rich lines and the algebraic language of symmetry sets, then prove Theorem~\ref{thm:4}.
In the next subsection, we prove Theorems~\ref{thm:3}, \ref{thm:5}, and \ref{thm:2}.
Theorem~\ref{thm:9x} is proved in the final subsection.

\subsection{A geometric/algebraic dictionary and proof of Theorem~\ref{thm:4}}
\label{sec:geom-dict-proof}

As we have said, $G=\aff 1\F$ consists of transformations $x\mapsto ax+b$ with $a,b\in\F$ and $a\not=0$.
The group $G$ acts on the \emph{affine line}\/ $X=\F$ by linear maps.
If $g\in G$, $Y$ is a finite subset of $X$, and $|Y\cap gY|\geq\alpha|Y|$, we say that $g$ is an \emph{$\alpha$-approximate symmetry of $Y$}.
The collection of all $\alpha$-approximate symmetries of a set is called a \emph{symmetry set}
\[
\sym_\alpha(Y) = \{g\in G\colon |Y\cap gY|\geq\alpha|Y|\}.
\]
Symmetry sets were first defined in additive combinatorics in \cite[Section 2.7]{tao2010additive};
symmetry sets for a general action of a group $G$ on a set $X$ are discussed in more detail in \cite{murphy2017group}.

Every affine transformation in $\aff 1\F$ corresponds to a line (its graph) in $\F^2$.
By convention, we ignore vertical lines, thus every line in $\F^2$ is the graph of transformation in $\aff 1\F$.

Several properties of rich lines correspond to properties of approximate symmetries:
\begin{enumerate}
  \item collections of rich lines in grids correspond to symmetry sets,
  \item collections of \emph{parallel lines}\/ correspond to cosets of the \emph{translation subgroup},
  \item collections of \emph{concurrent lines}\/ correspond of cosets of \emph{homothety subgroups}.
\end{enumerate}
To prove (1), simply note that if a line $\ell$ has the equation $y=ax+b$ then 
\begin{equation}
  \label{eq:1x}
  |\ell\cap (Y\times Y)|\geq\alpha|Y|\iff |Y\cap (aY+b)|\geq \alpha|Y|.
\end{equation}

To prove (2) and (3), we need a bit of background on the subgroups of the affine group.

Let $\tau_b$ denote the transformation $x\mapsto x+b$.
The \emph{translation subgroup}\/ $U:=\{\tau_b\colon b\in\F\}$ is a normal subgroup of $G$ corresponding to translations of $\F$.
($U$ is for ``unipotent''.)

Let $d_a$ denote the transformation $x\mapsto ax$.
The \emph{dilation subgroup}\/ $T=\{d_a\colon a\in\F^*\}$ corresponds to dilations of $\F$ about $0$.
In general, the stabilizer of a point $x$ in $\F$ under the action of $\aff 1\F$ has the form $\stab(x)=gTg^{-1}$, where $g(0)=x$. We call $\stab(x)$ the \emph{homothety subgroup}\/ of dilations about $x$.

The dilation subgroup and the homothety subgroups are the \emph{maximal abelian subgroups}\/ of $G$.
(If $H$ is abelian, either $H\subseteq U$ or there is an element $x\in H\setminus U$, and $H$ is contained in the centralizer of $x$, which is a homothety subgroup.)
We will usually say ``abelian subgroup'' rather than saying ``dilation or homothety subgroup''.

For $x,y\in\F$, the set of transformations $\trans(x,y)$ sending $x$ to $y$ has the form $\trans(x,y)=gTh$, where $h(y)=0$ and $g(0)=x$ ($g$ and $h$ are not unique).
We call $\trans(x,y)$ the \emph{transporter of $x$ to $y$}; it is a left coset of $\stab(x)$ and a right coset of $\stab(y)$.

If $L$ is a set of (non-vertical) lines in $\F^2$, let $A_L$ denote corresponding set of affine transformations.
\begin{itemize}
\item Property (2) holds since the lines of lines $L$ have common slope $a$ if and only if the corresponding set of affine transformations $A_L$ is contained in $d_aU$, and
\item Property (3) holds since the lines of $L$ are incident to a common point $(x,y)$ in $\F^2$ if and only if $A_L$ is contained in $\trans(x,y)$, which is a coset of a homothety subgroup.
\end{itemize}

Now we derive Theorem~\ref{thm:4} from Theorem~\ref{thm:9x}.
\begin{proof}[Proof of Theorem~\ref{thm:4}]
  Let $L$ be a set of $\alpha$-rich lines in $Y\times Y$ and let $A$ denote the set of affine transformations corresponding to $L$.

By Theorem~\ref{thm:9x}, if $(\alpha/2)^{2^J}\geq 1/|Y|$, and $|Y|\leq (\alpha/2)^{2^J}p$ in the case $\F=\F_p$, then there is an abelian subgroup $S\leq G$ and an element $g$ in $G$ such that
\[
|A\cap gS|\gg \left( \frac\alpha2 \right)^{C\, 2^J}|Y|^{-C/J}|A|.
\]

Let $L'$ denote the set of lines in $L$ that correspond to elements of $A\cap gS$.
By Properties 2 and 3, the lines of $L'$ are either parallel or concurrent, and since $|L'|=|A\cap gS|$ and $|L|=|A|$, the desired lower bound holds.
\end{proof}

\subsection{Proof of Theorems~\ref{thm:3}, \ref{thm:5}, and \ref{thm:2}}
\label{sec:proof-theor-refthm:3}

In this section we prove Theorems~\ref{thm:3}, \ref{thm:5}, and \ref{thm:2} using Theorem~\ref{thm:4}.
The proofs of Theorems~\ref{thm:3} and \ref{thm:5} simply consist of choosing parameters and checking that the hypotheses of Theorem~\ref{thm:4} are satisfied.
The proof of Theorem~\ref{thm:2} is essentially the same as the proof of Corollary~\ref{cor:4} presented in the introduction.

\begin{proof}[Proof of Theorem~\ref{thm:3}]
Let $N=|Y|$.
Let $J$ be a positive integer such that $J> 2C/\gamma\epsilon$, where $C$ is the constant from Theorem~\ref{thm:4}.
Choose $\delta = 1/(J 2^J)$.

To apply Theorem~\ref{thm:4} for $\alpha=N^{-\delta}$, we must check the constraints on $\alpha$ and $J$.
Since
\begin{equation}
  \label{eq:29}
  \left( \frac\alpha2 \right)^{2^J} = \left( \frac 1{2N^\delta} \right)^{2^J} = \frac 1{2^{2^J}N^{1/J}},
\end{equation}
for $N$ sufficiently large, we have $(\alpha/2)^{2^J}\geq 1/N= 1/|Y|$.
If $\F=\F_p$, we must check the additional constraint $|Y|\leq (\alpha/2)^{2^J}p$.
Since
\[
\left( \frac \alpha2 \right)^{2^J}p \gg_{\gamma,\epsilon} N^{-\gamma\epsilon/2C}p,
\]
the additional constraint follows from the addition hypothesis $N\ll p^{1-O(\gamma\epsilon)}$ when $\F=\F_p$.

Thus in either case, we may apply Theorem~\ref{thm:4} to find a subset $L'\subseteq L$ of either parallel or concurrent lines such that
\[
|L'|\gg \left( \frac\alpha2 \right)^{C\, 2^J} |Y|^{-C/J}|L| \gg_J N^{-C\, \delta\, 2^J}N^{-C/J}|L|.
\]
To complete the proof, we must show that $|L'|\gg |L|^{1-\gamma}$, which follows from our choice of $J$ and $\delta$:
\[
N^{C\,\delta\, 2^J}N^{C/J} \ll N^{\gamma\epsilon} \leq |L|^{\gamma}.
\]
\end{proof}

\begin{proof}[Proof of Theorem~\ref{thm:5}]
  Let $N=|Y|$ and set
\[
J = \log_2 \left( \frac{\log_2 N}{\log_2\log_2 N} \right).
\]
Then $N^{1/2^J}=\log_2 N$, so $(\alpha/2)^{2^J}\geq 1/|Y|$ for $N$ sufficiently large.
Since
\begin{equation}
  \label{eq:10xx}
  \left( \frac\alpha2 \right)^{2^J} = N^{-\log_2(2/\alpha)/\log_2\log_2 N},
\end{equation}
the constraint $|Y|\leq (\alpha/2)^{2^J}p$ follows from the condition
\[
N^{1-\log_2(2/\alpha)/\log_2\log_2 N}\leq p.
\]
Thus we may apply Theorem~\ref{thm:4} to find a subset $L'\subseteq L$ of either parallel or concurrent lines such that
\[
|L'|\gg \left( \frac\alpha2 \right)^{C\, 2^J} |Y|^{-C/J}|L|.
\]

Since the lines of $L$ are in general position, we have $|L'|\leq 2$.
Thus
\[
|L| \ll \left( \frac 2\alpha \right)^{C\, 2^J} N^{C/J}.
\]
For $N$ sufficiently large, $J\gg \log_2\log_2 N$, by (\ref{eq:10xx}) we have
\[
|L| \ll N^{C \frac{1-\log(\alpha)}{\log\log N}}.
\]
\end{proof}

\begin{proof}[Proof of Theorem~\ref{thm:2}]
  Suppose that (\ref{eq:5}) is false.
Let $Y= AC\cup (A+B)$.
Then $|Y|\leq K|A|$.

Let $L$ denote the set of lines of the form $y=c(x-b)$ with $b\in B$ and $c\in C$.
Each line $\ell$ in $L$ satisfies $|Y\cap \ell(Y)|\geq |A| \geq \frac 1K |Y|$, hence $L$ is a set of $\alpha$-rich lines in $Y\times Y$ with $\alpha=1/K$.

The constraints on $K$ imply that $(\alpha/2)^{2^J}\geq 1/|Y|$ and $|Y|\leq (\alpha/2)^{2^J}p$, if $\F=\F_p$.
Thus by Theorem \ref{thm:9x}, there is subset $L'\subseteq L$ consisting of either parallel or concurrent lines with size
\[
|L'| \gg \left( \frac\alpha2 \right)^{C 2^J} |Y|^{-C/J} |L|.
\]

Since $L$ contains at most $|B|$ parallel lines and at most $|C|$ concurrent lines, we have
\[
\max(|B|,|C|) \gg \left( \frac\alpha2 \right)^{C 2^J} |Y|^{-C/J} |B||C|,
\]
hence
\[
\min(|B|,|C|) \ll (2K)^{C 2^J} |Y|^{C/J}.
\]
\end{proof}

\begin{remark}
  Theorem~\ref{thm:2} can be proved directly from Theorem~\ref{thm:9x} by noting that the transformations $x\mapsto c(x-b)$ are contained in $\sym_\alpha(Y)$ for $\alpha=1/K$.
\end{remark}

\subsection{Proof of Theorem~\ref{thm:9x}}
\label{sec:proof-theor-refthm:9}

Theorem~\ref{thm:9x} follows from a general inverse theorem for groups actions, which is a group action version of (asymmetric) \bsg{} theorem \cite{murphy2017group}.
In addition to this general inverse theorem, we need two other inputs specific to the action of $\aff 1\F$ on $\F$ for $\F=\C$ and $\F=\F_p$: 
\begin{enumerate}
\item a \emph{product theorem}\/ for $\aff 1\F$, and 
\item \emph{bounds}\/ for the size of $\sym_\alpha(Y)$.
\end{enumerate}

\subsubsection{Group action version of the (asymmetric) \bsg{} theorem}

First, we state the group action version of the \bsg{} theorem from \cite{murphy2017group}.
We simplify the statement slightly, and specialize to $\aff 1\F$ acting on $\F$.
\begin{thm}
\label{thm:1}
There is an absolute constant $C>0$ such that the following holds.

Let $Y$ be a finite subset of $\F$ and let $A$ be a finite subset of $\aff 1\F$.
Given a number $0<\alpha<1$ and an integer $J\geq 0$, define
\begin{equation}
  \label{eq:1}
\alpha_J=2 \left( \frac\alpha2  \right)^{2^J}\andd  K= \left( \frac{|\sym_{\alpha_J}(Y)|}{|A|} \right)^{1/J}.
\end{equation}

If $A\subseteq\sym_\alpha(Y)$, then
\begin{enumerate}
\item there is an element $g_*$ in $G$ and a finite subset $A_*\subseteq G$ such that
\begin{equation}
  \label{eq:2}
g_*^{-1}  A_*\subseteq \sym_{\alpha_J}(Y)
\end{equation}
and
\begin{equation}
  \label{eq:3}
  |A_*^3| \ll \left( \frac{K}{\alpha_J} \right)^C|A_*|,
\end{equation}
\item for any subset $S\subseteq G$ there is an element $g$ in $G$ such that
\begin{equation}
  \label{eq:4}
  |A\cap gS| \gg \left( \frac{\alpha_J}{K} \right)^C \frac{|S\cap A_*|}{|A_*|} |A|.
\end{equation}
\end{enumerate}
\end{thm}
Part (1)  of Theorem~\ref{thm:1} says that some symmetry set of $Y$ contains a set $A_*$ with small tripling, which will allow us to apply the product theorems, stated next, to find a coset $S$ of an abelian subgroup such that $|A_*\cap S|$ is large.
Part (2) of Theorem~\ref{thm:1} then says that $|A\cap gS|$ is large as well, which gives us the desired structure in $A$.

\subsubsection{Product theorems for $\aff 1\C$ and $\aff 1{\F_p}$}

The following product theorem is a special case of a product theorem for solvable groups of $GL_n(\C)$, due to  Breuillard and Green \cite[Theorem 1.4']{breuillard2011approximate-b}.
\begin{thm}[Product theorem for $\aff 1\C$]
  \label{thm:12}
Fix $K\geq 1$.
If $A$ is a finite subset of $\aff 1\C$ such that $|A^3|\leq K|A|$, there is a subset $A'\subseteq A$ with size $|A'|\geq K^{-C}|A|$ that is contained in a coset of an abelian subgroup of $\aff 1\C$.
\end{thm}

Over $\F_p$, Helfgott has proved a similar theorem \cite[Proposition 4.8]{helfgott2015growth}.
\begin{thm}[Product theorem for $\aff 1{\F_p}$]
  \label{thm:13}
Let $G=\aff 1{\F_p}$, let $U$ be the translation subgroup, and let $\pi\colon G\to G/U$ be the quotient map.

For a subset $A\subseteq G$, if there is a constant $K\geq 1$ such that $|A^3|\leq K|A|$, then for an absolute constant $C>0$ we have either
\begin{equation}
  \label{eq:33}
  |A\cap T|\geq \frac 13|A|
\end{equation}
for some torus $T$,
\begin{equation}
  \label{eq:34}
  |\pi(A)|\ll K^C,
\end{equation}
or
\begin{equation}
  \label{eq:35}
  K^C|A|\gg |\pi(A)|p.
\end{equation}
\end{thm}
Theorems~\ref{thm:12} and \ref{thm:13} can be proved by combining the orbit-stabilizer theorem for sets \cite[Lemma 4.1]{helfgott2015growth} with a pivot argument or sum-product theorem. 
For completeness, we include proofs of Theorems~\ref{thm:12} and \ref{thm:13} in Appendix~\ref{sec:prod-theor-affine}, using the sum-product theorems from \cite{roche-newton2016sum-product,yazici2015growth}.

Since $|\pi(A)|$ is the number of cosets of $U$ needed to cover $A$, if (\ref{eq:34}) holds, then there is an element $g$ in $G$ such that $|A\cap gU|\gg K^{-C}|A|$.
We also know that $|A|\ll |\sym_{\alpha_J}(Y)|$, and we will use this to draw a similar conclusion from (\ref{eq:35}) using the upper bounds for $|\sym_\alpha(Y)|$.

\subsubsection{Upper bounds for $|\sym_\alpha(Y)|$}

Finally, we quote upper bounds for the symmetry sets for the action of $\aff 1\F$ on $\F$.
\begin{thm}
  \label{thm:7}
Let $Y$ be a finite subset of $\F$ and let $\alpha$ be greater than $2/|Y|$.

If $\F=\C$, then
\[
|\sym_\alpha(Y)|\ll \alpha^{-3}|Y|.
\]
If $\F=\F_p$ and $|Y|\leq \frac\alpha2 p$, then
\[
|\sym_\alpha(Y)|\ll \alpha^{-4}|Y|.
\]
\end{thm}
See \cite{murphy2017group} for a proof of Theorem~\ref{thm:7},  which is based on the \ST{} theorem \cite{szemeredi1983extremal,toth2015szemeredi-trotter,zahl2015szemeredi-trotter, solymosi2007number} for $\F=\C$ or the Stevens-de Zeeuw bound \cite{stevens2017improved} combined with some additional arguments \cite{murphy2017results} for $\F=\F_p$.

\begin{remark}
  Weaker bounds than those of Theorem~\ref{thm:7} suffice for the proof of Theorem~\ref{thm:9x}.
We give specifics after the proof.
This is in constrast to Elekes' proof of Theorem~\ref{thm:22}, which depends crucially on having bounds for $|\sym_\alpha(Y)|$ that are \emph{linear}\/ in $|Y|$.
\end{remark}

\subsubsection{Proof of Theorem~\ref{thm:9x}}

\begin{proof}[Proof of Theorem~\ref{thm:9x}]
  The condition $(\alpha/2)^{2^J}\geq 1/|Y|$ implies that $\alpha_J\geq 2/|Y|$, and the condition $|Y|\leq (\alpha/2)^{2^J}p$ implies that $|Y|\leq \frac 12\alpha p$.
Hence by Theorem~\ref{thm:7} we have
\begin{equation}
  \label{eq:14x}
  K \leq |\sym_{\alpha_J}(Y)|^{1/J} \ll \left( \frac\alpha 2 \right)^{-C} |Y|^{1/J}.
\end{equation}

By Theorem~\ref{thm:1}, there is a constant $C>0$, an element $g_*$ in $G$, and a subset $A_*$ of $g_*\sym_{\alpha_J}(Y)$ such that
\begin{equation}
  \label{eq:40}
  |A_*^3|\ll (\alpha_J^{-1}K)^{C}|A_*|.
\end{equation}

Now, suppose that $\F=\C$.
By (\ref{eq:40}) and Theorem~\ref{thm:12}, there is an element $g$ in $G$ and an abelian subgroup $H$ of $G$ such that
\[
|A_*\cap gH|\gg (\alpha_J^{-1}K)^{-C}|A_*|.
\]
By equation~(\ref{eq:4}) of Theorem~\ref{thm:1}, there is an element $g'$ in $G$ such that
\begin{equation}
  \label{eq:10x}
  |A\cap g'gH|\gg \alpha_J^2 (\alpha_J^{-1}K)^{-C}\frac{|A_*\cap gS|}{|A_*|} |A_0| \gg \alpha_J^C |Y|^{-C/J}|A_0|.
\end{equation}
Since $\alpha_J=2(\alpha/2)^{2^J}$, the proof is complete.

If $\F=\F_p$, then we apply Theorem~\ref{thm:13} in place of Theorem~\ref{thm:12}.
If (\ref{eq:33}) or (\ref{eq:34}) hold, then the proof is the same as in the case of $\F=\C$, so suppose that (\ref{eq:35}) holds:
\begin{equation}
  \label{eq:9x}
  \left( \frac K{\alpha_J} \right)^C|A_*|\gg |\pi(A_*)|p.
\end{equation}
Since $A_*\subseteq g_*\sym_{\alpha_J}(Y)$, by Theorem~\ref{thm:7} we have 
\begin{equation}
  \label{eq:13x}
  |A_*|\leq |\sym_{\alpha_J}(Y)| \ll \alpha_J^{-4}|Y| \ll \alpha_J^{-3}p.
\end{equation}
Combining this with (\ref{eq:9x}) we have
\[
|\pi(A_*)| \ll \left( \frac K{\alpha_J} \right)^C,
\]
which implies that there is an affine transformation $g$ such that
\[
|A_*\cap gU| \gg\left( \frac{\alpha_J}K \right)^C|A_*|.
\]
The rest of the proof is the same as in (\ref{eq:10x}).
\end{proof}

\begin{remark}
  Instead of using Theorem~\ref{thm:7} to prove (\ref{eq:14x}), we could have used the bound $|\sym_\alpha(Y)|\ll \alpha^{-2}|Y|^2$, which follows from the Cauchy-Schwarz inequality and holds for $\aff 1\F\actson \F$ for any field $\F$, or even the trivial bound $|\sym_\alpha(Y)|\leq |Y|^4$, which holds because $|Y|^2$ points support at most $|Y|^4$ lines containing at least two elements of the point set.

Equation~(\ref{eq:13x}) could be proved using Vinh's incidence bound \cite{vinh2014point-line}, which can also be proved using only Cauchy-Schwarz \cite{murphy2016point-line}.
\end{remark}





\appendix 

\section{Proof of group action \bsg{}}
\label{sec:proof-group-action-bsg}

In this section, we sketch the proof of the group action version the asymmetric \bsg{} theorem, which we recall here.
This theorem is proved in more generality (and in full detail) in \cite{murphy2017group}.
\begin{repthm}{thm:1}
There is an absolute constant $C>0$ such that the following holds.

Let $Y$ be a finite subset of $\F$ and let $A$ be a finite subset of $\aff 1\F$.
Given a number $0<\alpha<1$ and an integer $J\geq 0$, define
\begin{equation}
  \label{eq:1x}
\alpha_J=2 \left( \frac\alpha2  \right)^{2^J}\andd  K= \left( \frac{|\sym_{\alpha_J}(Y)|}{|A|} \right)^{1/J}.
\end{equation}

If $A\subseteq\sym_\alpha(Y)$, then
\begin{enumerate}
\item there is an element $g_*$ in $G$ and a finite subset $A_*\subseteq G$ such that
\begin{equation}
  \label{eq:2x}
g_*^{-1}  A_*\subseteq \sym_{\alpha_J}(Y)
\end{equation}
and
\begin{equation}
  \label{eq:3x}
  |A_*^3| \ll \left( \frac{K}{\alpha_J} \right)^C|A_*|,
\end{equation}
\item for any subset $H\subseteq G$ there is an element $g$ in $G$ such that
\begin{equation}
  \label{eq:4x}
  |A\cap gH| \gg \left( \frac{\alpha_J}{K} \right)^C \frac{|H\cap A_*|}{|A_*|} |A|.
\end{equation}
\end{enumerate}
\end{repthm}

To understand how our method works, we will first revisit Elekes' proof of Theorem~\ref{thm:22}.
The key idea is that \emph{symmetry sets behave weakly like groups}.
In fact, $\sym_1(Y)$ is a group: it is the stabilizer of $Y$ under the induced action of on subsets of $X$.
For $\alpha<1$, a weak form of multiplicative closure holds.
\begin{prop}[Approximate multiplicative closure]
  \label{prop:300}
If $S$ is a non-empty subset of $\sym_\alpha(Y)$, then there exists a relation $E\subseteq S^{-1}\times S$ such that
\[
|E|\geq\frac{\alpha^2}2|S|^2\andd S^{-1}\pp{E}S\subseteq\sym_{\frac{\alpha^2}2}(Y).
\]
Further, $(S^{-1}\pp{E}S)^{-1}=S^{-1}\pp{E}S$.
\end{prop}
This is \cite[Proposition 3]{murphy2017group}, which is a straightforward generalization of \cite[Lemma 2.33]{tao2010additive}, which follows easily from Cauchy-Schwarz.

To prove that Theorem~\ref{thm:22} follows from a product theorem, such as Theorem~\ref{thm:6}, we combine Proposition~\ref{prop:300} with the upper bounds of Theorem~\ref{thm:7}.
\begin{prop}
  \label{prop:7xx}
Let $\F$ be a field, and let $G=\aff 1\F$ act on $X=\F$ by affine transformations.
Let $A\subseteq G$ and $Y\subseteq X$ be finite subsets such that $A\subseteq\sym_\alpha(Y)$ and $|A|\geq |Y|$.
Then there is a subset $E\subseteq A\times A$ such that $|E|\geq\frac{\alpha^2}2 |A|^2$ and
\begin{enumerate}
\item if $\F=\C$, then $|A^{-1}\pp E A|\ll \alpha^{-6}|A|$,
\item if $\F=\F_p$ and $|Y|\leq \frac 12\alpha |Y|$, then $|A^{-1}\pp E A|\ll \alpha^{-8}|A|$.
\end{enumerate}
\end{prop}
\begin{proof}
  By Proposition~\ref{prop:300}, there is a subset $E\subseteq A\times A$ such that $|E|\geq \frac{\alpha^2}2|A|^2$ and 
\[
A^{-1}\pp E A\subseteq\sym_{\alpha^2/2}(Y).
\]
By Theorem~\ref{thm:7}, if $\F=\C$,
\[
|A^{-1}\pp E A|\leq|\sym_{\alpha^2/2}(Y)|\ll \alpha^{-6}|Y|\leq \alpha^{-6}|A|,
\]
while if $\F=\F_p$ and $|Y|\leq\frac 12\alpha p$,
\[
|A^{-1}\pp E A|\leq|\sym_{\alpha^2/2}(Y)|\ll \alpha^{-8}|Y|\leq \alpha^{-8}|A|.
\]
\end{proof}

Now we will prove the following theorem, in the spirit of Elekes' Theorem~\ref{thm:22}.
\begin{thm}
  \label{thm:8x}
If $N$ lines are $\alpha$-rich in a $N\times N$ grid in $\C^2$,
then either
\begin{enumerate}
\item $C\alpha^C N$ lines are parallel, or
\item $C\alpha^C N$ lines are concurrent,
\end{enumerate}
where $C>0$ is a constant independent of $\alpha$ and $N$.

The same holds for $\F_p^2$, provided that $N\leq\frac 12\alpha p$.
\end{thm}
To prove Theorem~\ref{thm:8x}, we need the \bsg{} theorem.
The following version \cite[Lemma 34]{murphy2017group}, is essentially contained in \cite{tao2008product}.
\begin{thm}
  \label{thm:8}
If $A$ and $B$ are finite subsets of a group $G$ and $E\subseteq A\times B$ is a relation such that
\[
|E|\geq \alpha|A||B| \andd |A\pp E B|\leq K|A|^{1/2}|B|^{1/2},
\]
where $\alpha\in (0,1]$ and $K>0$, then there is an element $a$ in $A$ and a subset $S\subseteq a^{-1}A$ such that
\[
|S|\gg \left( \frac \alpha K \right)^C |A|\andd |S^3|\ll \left( \frac K\alpha \right)^C|S|,
\]
where $C$ is an absolute constant.
\end{thm}
\begin{proof}
Let $\F$ denote $\C$ or $\F_p$, let $Y$ be a finite subset of $\F$, and suppose that $L$ is a set of lines that are $\alpha$-rich in $Y\times Y$.
Translating to algebraic language, we have $A\subseteq\sym_\alpha(Y)$, where $A=A_L$ is the set of affine transformations corresponding to the elements of $L$.

By Proposition~\ref{prop:300}, there is a subset $E\subseteq A^{-1}\times A$ such that
\[
|E|\geq\frac{\alpha^2}2|A|^2\andd |A^{-1}\pp E A|\ll \alpha^{-8}|A|.
\]
By Theorem~\ref{thm:8}, there is an element $a$ of $A$ and a subset $S$ of $\aff 1\F$ such that $S\subseteq aA^{-1}$,
\[
|S|\gg \alpha^{-C}|A|,\andd |S^3|\ll \alpha^{-C}|S|.
\]

Now, as in the proof of Theorem~\ref{thm:9x}, we may apply Theorem~\ref{thm:12} or Theorem~\ref{thm:13}, depending on $\F=\C$ or $\F=\F_p$, to deduce that there is an abelian subgroup $H$ of $\aff 1\F$ such that $|S\cap gH|\gg\alpha^{-C}|S|$ for some $g$ in $\aff 1\F$.
Since $S\subseteq aA^{-1}$ and $|S|\gg\alpha^{-C}|A|$, we have $|aA^{-1}\cap gH|\gg\alpha^{-C}|A|$, hence $|A\cap g^{-1}aH'|\gg\alpha^{-C}|A|$ for some subgroup $H'$ conjugate to $H$.

To complete the proof, we translate to back geometric language, as in the proof of Theorem~\ref{thm:4}.
\end{proof}
It is a credit to Elekes' ingenuity that he proved Theorem~\ref{thm:22} without the \bsg{} theorem.

The assumption $|A|\geq |Y|$ is necessary to compare $|A^{-1}\pp E A|$ to $|A|$; if $|A| < |Y|$, then one may iterate Proposition~\ref{prop:300} until we reach an iterated partial product set with small doubling.
This strategy was used by Borenstein and Croot \cite{borenstein2010lines} to prove an analog of Elekes' results for small sets of lines (affine transformations).
The analogy between Borenstein and Croot's work \cite{borenstein2010lines} and the asymmetric \bsg{} theorem \cite[Theorem 2.35]{tao2010additive}, as observed by Helfgott \cite{borenstein2010lines}, motivated Theorem~\ref{thm:1}.

To prove Theorem~\ref{thm:1}, we use a variation of Proposition~\ref{prop:300}.
We use the notation $\lessapprox$ to hide logarithmic factors of $\alpha^{-1}$ and $|A|$, and for finite subsets $A$ and $B$ of a group and $E\subseteq A\times B$ we define
\[
 A\pp{E}B=\{ab\colon (a,b)\in E\}\andd r_E(x)=|\{(a,b)\in E\colon ab=x\}|.
\]
\begin{prop}[Uniform approximate closure]
  \label{prop:1}
If $A$ is a non-empty subset of $\sym_\alpha(Y)$ then there is a relation $E\subseteq A^{-1}\times A$ such that
\begin{align}
  |E| &\gtrapprox \alpha^2 |A|^2,\label{eq:54x}\\
r_E(x) &\geq\frac{|E|}{2|A^{-1}\pp{E}A|}\qquad\mbox{for all $x$ in $A^{-1}\pp{E}A$},\label{eq:55x}\\
& A^{-1}\pp{E} A \subseteq \sym_{\frac{\alpha^2}2}(Y)\label{eq:56x}.
\end{align}
\end{prop}
The proof of Proposition~\ref{prop:1} is essentially the same as the proof of \cite[Lemma 2.34]{tao2010additive}: combine Proposition~\ref{prop:300} with a dyadic pigeonholing argument.
(While \cite[Lemma 2.34]{tao2010additive} is stated only for abelian groups, the proof works verbatim for non-abelian groups.)

Proposition~\ref{prop:1} implies that if a set $S$ is dense in the product set $A^{-1}\pp E A$, then some translate of $S$ is dense in $A$.
Thus, if we find a ``structured'' subset of the product set $A^{-1}\pp E A$, we may bring that structure back to the original set $A$.
More precisely, if $A$ is a finite subset of $G$ and $E\subseteq A^{-1}\times A$ satisfies (\ref{eq:54x}) and (\ref{eq:55x}), then for any subset $S$ of $G$, there is an element $a$ in $A$ such that
\begin{equation}
  \label{eq:7}
  \frac{|A\cap aS|}{|A|}\gtrapprox\alpha^2\frac{|(A^{-1}\pp E A)\cap S|}{|A^{-1}\pp E A|}.
\end{equation}
Now we sketch the proof of Theorem~\ref{thm:1}.
\begin{proof}[Proof of Theorem~\ref{thm:1}]
Let $A_0=A$ and $\alpha_0=\alpha$.
By Proposition~\ref{prop:1}, there is a subset $E_0\subseteq A_0^{-1}\times A_0$ such that \eqref{eq:54x}, \eqref{eq:55x}, and \eqref{eq:56x} hold.
Define $A_1:=A_0^{-1}\pp{E_0}A_0$.
By \eqref{eq:7}, for any subset $S$ of $\aff 1\F$, there is an element $a_0$ in $A_0$ such that
\[
\frac{|A_0\cap a_0S|}{|A_0|}\gtrapprox \alpha_0^2\frac{|A_1\cap S|}{|A_1|}.
\]

Since $A_1\subseteq\sym_{\alpha_1}(Y)$, where $\alpha_1=\alpha_0^2/2$, we may iterate this process to find a sequence of numbers
\[
\alpha_0>\alpha_1>\cdots>\alpha_J>0
\]
such that $\alpha_{j+1}=\alpha_j^2/2$, and a sequence of sets $A_j\subseteq\aff 1\F$ such that $A_j\subseteq\sym_{\alpha_j}(Y)$,
and for any set $S$ in $\aff 1\F$, there is an element $a_j$ in $A_j$ such that
\begin{equation}
  \label{eq:8}
  \frac{|A_j\cap a_jS|}{|A_j|}\gtrapprox \alpha_j^2\frac{|A_j\cap S|}{|A_j|}.
\end{equation}

Now for the key step: setting $K^J=|A_J|/|A_0|$, we have
\[
K^J=\frac{|A_J|}{|A_0|}=\prod_{j=0}^{J-1}\frac{|A_{j+1}|}{|A_j|},
\]
so by the pigeonhole principle, there is an index $0\leq j\leq J-1$ such that $|A_{j+1}|\leq K|A_j|$.
That is,
\[
|A_j^{-1}\pp{E_j}A_j|\leq K|A_j|.
\]
Since $|E_j|\gtrapprox \alpha_j^2|A_j|$, we can now apply the \bsg{} theorem, as in the proof of Theorem~\ref{thm:8x}, to find a subset $S$ of $a_jA_j^{-1}$ such that
\[
|S\cap a_jA_j^{-1}|\gg \left( \frac{\alpha_j}{K} \right)^C|A_j|\andd |S^3|\ll \left( \frac {K}{\alpha_j} \right)^C|S|.
\]

If we wished to prove Theorem~\ref{thm:9x} directly, we would now apply a product theorem to find an abelian subgroup of $\aff 1\F$ with large overlap with $S$.

Instead, we simply assume that there is some set $H$ such that $|S\cap H|\gg (\alpha_j/K)^C|S|$.
Since $S\subseteq a_jA_j^{_1}$, we have $|a_j^{-1}H\cap A_j^{-1}|\gg (\alpha_j/K)^C|A_j|$.
Iterating \eqref{eq:8} yields an element $g$ in $G$ such that
\[
  \frac{|A_0\cap ga_j^{-1}H|}{|A_0|}\gtrapprox(\alpha_0\cdots\alpha_j)^2\frac{|A_j\cap a_j^{-1}H|}{|A_j|}.
\]
Since $\alpha_j=2(\alpha/2)^{2^j}\geq 2(\alpha/2)^{2^J}$, this completes the proof of Theorem~\ref{thm:1}.
\end{proof}

\section{Product theorems for $\aff 1\F$}
\label{sec:prod-theor-affine}

Let $U$ be a subgroup of a group $G$, and let $\pi\colon G\to G/U$ be the quotient map.
For a subset $A$ of $G$, let $A/U$ denote the image of $A$ under $\pi$; that is, $A/U$ is the set of left cosets of $U$ of the form $aU$ with $a$ in $A$.

Recall that if $G=\aff 1\F$, then a maximal torus $T$ is a subgroup conjugate to the diagonal subgroup, and the unipotent subgroup $U$ consists of upper triangular matrices with 1's on the diagonal.
Every abelian subgroup of $\aff 1\F$ is either contained in the unipotent subgroup $U$ or a maximal torus.

The following is a specialization of \cite[Theorem 1.4']{breuillard2011approximate-b} to $\aff 1\C$. 
\begin{thm}[Product theorem for $\aff 1\C$]
\label{thm:17}
If $A$ is a subset of $\aff 1\C$ such that $|A^3|\leq K|A|$, then either $\geq |A|/3$ elements of $A$ are contained in a torus, or
\[
K^{10}|A|\gg |A/U|^{1/2}|A|,
\]
hence there is an element $g$ in $G$ such that $|A\cap gU|\gg K^{-20}|A|$.
\end{thm}
Theorem~\ref{thm:17} says that if $A$ is not contained in a torus, then either $A$ is covered by a small number of cosets of $U$ (so that $A/U$ is small), or $A$ grows under multiplication: $|A^3|\gg |A/U|^{1/20}|A|$. 

The next theorem is a slight quantitative improvement of the product theorem for $\aff 1{\F_p}$ that appears in \cite{helfgott2015growth}.
\begin{thm}[Product theorem for $\aff 1{\F_p}$]
\label{thm:18}
If $A$ is a subset of $\aff 1{\F_p}$ such that $|A^3|\leq K|A|$, then either $\geq |A|/3$ elements of $A$ are contained in a torus, or
\[
K^{10}|A|\gg |A/U|^{1/2}|A|,
\]
or
\[
K^{10}|A|\gg |A/U|p.
\]
\end{thm}

\subsection{Technical lemma}

The following lemma contains the common elements of the proofs of Theorem~\ref{thm:17} and Theorem~\ref{thm:18}.
\begin{lem}
\label{lem:6}
  If $\F$ is a field and $A$ is a finite subset of $\aff 1\F$, then either more than one third of the elements of $A$ are contained in an abelian subgroup, or there exists an $x$ in $A$ such that $|x^Ax^{-1}|=|[A,x]| > 1$.

Further there is an $a_0$ in $A$ such that if $S:=x^Ax^{-1}\subseteq U$ and $T:=(a_0^{-1}A)\cap C(x)$ then
\[
|S||T|\geq |A|\andd |T|\leq |A/U|.
\]
In addition, if $|A^3|\leq K|A|$, then
\[
K^{10}|A| \gg |A/U|\cdot |B-BC|,
\]
where $|B|=|S|$ and $|C|=|T|$.
\end{lem}
The proof of Lemma~\ref{lem:6} requires the following version of the orbit-stabilizer theorem \emph{for sets}, rather than for groups \cite{helfgott2015growth} and Ruzsa's triangle inequality.
\begin{lem}
\label{lem:3}
Suppose $G\actson X$, $x\in X$, and $A\subseteq G$ is finite.
Then there exists $a_0$ in $A$ such that
\begin{equation}
  \label{eqx1}
  |(a_0^{-1}A)\cap\stab(x)|\geq\frac{|A|}{|A(x)|},
\end{equation}
and for all finite sets $B\subseteq G$,
\begin{equation}
  \label{eqx2}
  |BA|\geq |A\cap\stab(x)||B(x)|.
\end{equation}
\end{lem}
\begin{prop}[Ruzsa triangle inequality]
\label{prop:14}
  If $A,B,C$ are finite subsets of a group, then
\[
|AC^{-1}|\leq\frac{|AB^{-1}||BC^{-1}|}{|B|}.
\]
\end{prop}

\begin{proof}[Proof of Lemma~\ref{lem:6}]
Suppose that at most $|A|/3$ elements of $A$ are contained in an abelian subgroup.  
Then at least $2|A|/3$ element of $A$ are not contained in the unipotent group $U$, so without loss of generality, we may assume that $A$ does intersect the unipotent group $U$.
(That is, we will use $A$ to denote $A\setminus U$.)

We still know that half of the elements of $A$ are not contained in an abelian subgroup, thus there exists an $x$ in $A$ such that
\begin{equation}
  \label{eqx17}
  |x^Ax^{-1}|=|[A,x]| > 1.
\end{equation}
Otherwise, $axa^{-1}x^{-1}=e$ for all $a,x$ in $A$, which implies that the subgroup generated by $A$ is abelian.

The set $x^A=\{axa^{-1}\colon a\in A\}$ is the orbit of $x$ under the action of $G$ on itself by conjugation; the stabilizer of $x$ is denoted $C(x)$.
Since $x\not\in U$, we know that $C(x)$ is conjugate to the diagonal subgroup of $\aff 1\F$; in particular, the only element of $U$ fixed by $C(x)$ under conjugation is the identity element.

By Lemma~\ref{lem:3}, there is an element $a_0$ in $A$ such that
\[
|(a_0^{-1}A)\cap C(x)|\geq\frac{|A|}{|x^A|}.
\]
Let $T$ denote $a_0^{-1}\cap C(x)$.

Now, if $S=[A,x] = x^A\cdot x^{-1}$, then $|S|=|x^A|$; by (\ref{eqx17}) we know $|S| > 1$, so $S$ contains an element of $U$ besides the identity.
Since $|T|=|(a_0^{-1}A)\cap C(x)|$, the previous equation can be restated as
\[
|S||T|\geq |A|.
\]

In addition, note that $|T|\leq |A/U|$, where $A/U$ is the image of $A$ under the quotient map $\pi\colon G\to G/U$.
The inequality $|T|\leq |A/U|$ follows since $\pi$ is injective when restricted to a torus.

Let $S^T$ denote the image of $S$ under the action of $T$ by conjugation.
Since $S\subseteq U$ and $U$ is preserved by conjugation, we have $S\cdot (S^T)^{-1}\subseteq U$.

Let
\[
B=\{z\colon  (\begin{smallmatrix} 1&z\\ 0&1 \end{smallmatrix})\in S, z\not=0\}
\]
and let
\[
C =\{a\colon\exists b  \,(\begin{smallmatrix} a&b\\ 0&1 \end{smallmatrix})\in T\}.
\]
Then
\[
|S\cdot (S^T)^{-1}\cap U|=|S\cdot (S^T)^{-1}|\geq |B-BC|,
\]
since conjugating  $(\begin{smallmatrix} 1&z\\ 0&1 \end{smallmatrix})$  by $(\begin{smallmatrix} a&b\\ 0&1 \end{smallmatrix})$ yields  $(\begin{smallmatrix} 1&az\\ 0&1 \end{smallmatrix})$.

Clearly, $|B|\geq |S|-1$ and since $\pi\colon G\to G/U$ is injective when restricted to $T$, we have $|C|=|T|$.

Note that
\[
S\cdot (S^T)^{-1}= x^Ax^{-1}((x^Ax^{-1})^T)^{-1} = x^Ax^{-1}(x^{TA}x^{-1})^{-1}=x^A(x^{-1})^{TA}.
\]
So that
\[
S\cdot (S^T)^{-1}\subseteq AxA^{-1}A^2x^{-1}A^{-2}\subseteq A^2A^{-1}A^2A^{-3}.
\]

By Lemma~\ref{lem:3}, we have
\[
|A^3A^{-1}A^2A^{-3}|\geq |(A^2A^{-1}A^2A^{-3})\cap U||A/U| \geq |S\cdot (S^T)^{-1}| |A/U|.
\]

By Proposition~\ref{prop:14}, if $|A^3|\leq K|A|$, then
\[
|A^3A^{-1}A^2A^{-3}| \leq K^{10}|A|.
\]

All together, we have
\[
K^{10}|A| \geq |S\cdot (S^T)^{-1}| |A/U| \geq |A/U||B-BC|,
\]
as desired.
\end{proof}
Note that if $B$ contains only $0$, then $|B-BC|=1$; where as if $B$ contains a non-zero element, then $|B-BC|\geq |C|$.

\subsection{Proof of results over $\C$}

The following theorem is an easy consequence of the \ST{} theorem, see \cite[Exercise 8.3.3]{tao2010additive}.
\begin{prop}
\label{prop:12}
If $A,B,C$ are finite subsets of $\C$, then
\[
|A+BC|\gg\sqrt{|A||B||C|}.
\]
\end{prop}

\begin{proof}[Proof of Theorem~\ref{thm:17}]
If at least one third of the elements in $H$ are contained in an abelian subgroup, then we are done.

Otherwise, by Lemma~\ref{lem:6} and Proposition~\ref{prop:12}, we have
\[
K^{10}|H| \gg |H/U| |B-BC|\gg |H/U| |B||C|^{1/2} = |H/U||S||T|^{1/2}.
\]
Since $|T|\leq |H/U|$ and $|S||T|\geq |H|$, we have
\[
K^{10}|H| \gg |H/U|^{1/2} |S||T| \geq |H/U|^{1/2}|H|.
\]
\end{proof}

\subsection{Proof of results over $\F_p$}
\label{sec:proof-Fp}

The following sum-product theorem is a slight improvement of a result of Roche-Newton, Rudnev, and Shkredov~\cite{roche-newton2016sum-product}, due to Stevens and de Zeeuw \cite[Corollary 10]{stevens2017improved}.
\begin{prop}
\label{prop:13}
If $A,B,C\subseteq\F_p$ where $p$ is prime, then
\[
|A+BC|\gg \min \left( \sqrt{|A||B||C|}, p \right).
\]
\end{prop}
In particular,
\begin{equation}
  \label{eqx3}
  |B\pm BC|\gg \min \left(|B||C|^{1/2},  p \right).
\end{equation}
\begin{proof}[Proof of the product theorem over $\F_p$]
If more than one third of $A$ is contained in an abelian subgroup, then we are done.

Otherwise, by Lemma~\ref{lem:6} and Proposition~\ref{prop:13}, we have
\[
K^{10}|A| \gg |A/U| |B-BC|\gg |A/U| \min \left( |B||C|^{1/2}, p \right).
\]

If the minimum is $|B||C|^{1/2}$, then as in the previous proof, we have
\[
K^{10}|A|\gg |A/U|^{1/2}|A|.
\]
If the minimum is $p$, then we have
\[
K^{10}|A|\gg |A/U|p.
\]
\end{proof}

\bibliographystyle{plain}
\bibliography{/Users/brendan/Dropbox/library}

\end{document}